\documentclass{amsart}

\usepackage{latexsym,amsmath,amssymb,amsfonts,amscd,graphics}
\usepackage[mathscr]{eucal}
\usepackage[all]{xy}

\newtheorem{theorem}{Theorem}[section]
\newtheorem{lemma}[theorem]{Lemma}
\newtheorem{proposition}[theorem]{Proposition}
\newtheorem{corollary}[theorem]{Corollary}

\theoremstyle{definition}

\newtheorem{example}[theorem]{Example}
\newtheorem{notation}[theorem]{Notation}

\theoremstyle{remark}
\newtheorem{remark}[theorem]{Remark}

\numberwithin{equation}{section}

\newcommand{\bP}{\mathbb{P}}
\newcommand{\bZ}{\mathbb{Z}}
\newcommand{\bC}{\mathbb{C}}
\newcommand{\bQ}{\mathbb{Q}}
\newcommand{\calO}{\mathscr{O}}

\newcommand{\calE}{\mathcal{E}}

\newcommand{\calB}{\mathcal{B}}
\newcommand{\calD}{\mathcal{D}}
\newcommand{\calH}{\mathcal{H}}
\newcommand{\gquot}{/\!\!/}
\DeclareMathOperator{\Pic}{Pic}

\DeclareMathOperator{\Proj}{Proj}
\DeclareMathOperator{\Chow}{Chow}
\DeclareMathOperator{\rank}{rank}
\DeclareMathOperator{\im}{Im}
\DeclareMathOperator{\divisor}{div}

\DeclareMathOperator{\SL}{SL}

\DeclareMathOperator{\SU}{SU}

\begin{document}
\title[Ball quotient for genus $4$]{The geometry of the ball quotient model of the moduli space of genus four curves}

\author[S. Casalaina-Martin]{Sebastian Casalaina-Martin}
\address{University of Colorado at Boulder, Department of Mathematics,   Boulder, CO 80309}
\email{casa@math.colorado.edu}
\thanks{The  first author was partially supported by NSF grant DMS-1101333.}

\author[D. Jensen]{David Jensen}
\address{Stony Brook University, Department of Mathematics,  Stony Brook, NY 11794}
\email{djensen@math.sunysb.edu}

\author[R. Laza]{Radu Laza}
\address{Stony Brook University, Department of Mathematics,  Stony Brook, NY 11794}
\email{rlaza@math.sunysb.edu}
\thanks{The  third author was partially supported by NSF grant DMS-0968968 and a Sloan Fellowship.}


\bibliographystyle{amsalpha}
\begin{abstract}
S. Kondo has constructed a ball quotient compactification   for the moduli space of non-hyperelliptic genus four curves. In this paper, we show that this space essentially coincides with a GIT  quotient of the Chow variety of  canonically embedded genus four curves.   More specifically, we give an explicit description of this GIT quotient, and show that the birational map from this space  to Kondo's space is resolved by the blow-up of a single point.   This provides a modular interpretation of the points in the boundary of Kondo's space.  Connections with the slope nine space in the Hassett-Keel program are also discussed.
\end{abstract}

\maketitle

\section*{Introduction}
Kondo \cite{k2} has constructed a ball quotient compactification  $(\mathcal B_9/\Gamma)^*$ of the moduli space of non-hyperelliptic genus four curves. In this paper, we discuss the relationship between this space and a GIT model of $\overline M_4$, the moduli space of genus four, Deligne-Mumford stable curves. To be precise, we construct a GIT quotient  $\overline M_4^{GIT}$ of canonically embedded genus four curves via a related GIT problem for cubic threefolds.  Results for cubic threefolds due to Allcock \cite{allcock1} allow us to  completely describe the stability conditions for $\overline M_4^{GIT}$ (Theorem \ref{thmcubic}).  With this, we can employ general results of Looijenga \cite{looijengacompact1} to give an explicit resolution of the period map $\overline M_4^{GIT} \dashrightarrow (\mathcal B_9/\Gamma)^*$ (Theorem \ref{mainthm}).   In addition, we identify $\overline M_4^{GIT}$ with a GIT quotient of the Chow variety of canonically embedded genus four curves 
(Theorem \ref{chowcubic}).   Some connections to the Hassett-Keel program are discussed in section \ref{sectHK}. In particular, we identify $\overline M_4^{GIT}$ with $\overline M_4(5/9)$, providing a modular interpretation for the slope $9$ log canonical model of $\overline M_4$ (Theorem \ref{thmhk}).

A standard method of constructing an algebraic moduli space is via a period map. When the period domain is a Hermitian symmetric domain $\calD$, one can in some cases find a period map that is generically injective and dominant. In these situations, the Baily--Borel compactification $(\calD/\Gamma)^*$ of the associated locally symmetric variety $\calD/\Gamma$ provides a projective model for the moduli space (where $\Gamma$ is the monodromy group).    The rich geometric structure of locally symmetric varieties provides a powerful tool for the study of these moduli spaces.  The main examples are the moduli spaces of abelian varieties, where the period domain is the Siegel space, and the moduli spaces of $K3$ surfaces, where the period domain is of Type $IV$.  Moduli spaces of Hyperk\"ahler manifolds, and the moduli space of cubic fourfolds, also have period maps to Type $IV$ domains that behave similarly to period maps for $K3$ surfaces.   Using these examples, and special constructions, it is sometimes possible to find moduli spaces birational to a ball quotient (see e.g. \cite{dk}). The few ball quotient examples known in the literature are the following: $n$  (weighted) points in $\bP^1$  with $3\le n \le 12$ (\cite{delignemostow}), curves of genus $g\le 4$ (\cite{kondo,k2}), del Pezzo surfaces of degree $d\le 3$ (\cite{act0}, \cite{dvk}, \cite{looiheck}), cubic threefolds (\cite{act,ls}), and some classes of Calabi-Yau threefolds (\cite{rohde}).  In the cases of interest here, namely genus four curves and cubic threefolds, the constructions use period maps of $K3$ surfaces and  cubic fourfolds, respectively.

 In general, it is difficult to determine what geometric objects should correspond to the boundary points in a moduli space obtained  as a locally symmetric variety.  For ball quotients and quotients of Type $IV$ domains, a now standard approach to this type of problem is to use a comparison with a moduli space constructed via GIT, and there is a well developed theory that covers this (see Looijenga \cite{looijengacompact1,looijengacompact} and Looijenga--Swierstra \cite{ls2}). In the case of genus $3$ curves, where the space of plane quartics provides a natural GIT compactification, the problem was completed by Looijenga \cite{l2} and Artebani \cite{artebani}.   From this perspective, our results give in the case of genus four curves, a modular  interpretation of the boundary points  of  Kondo's ball-quotient compactification via a GIT quotient of the Chow variety of canonically embedded genus four curves.

In the opposite direction, we point out that a moduli space of varieties constructed using GIT will in general include points corresponding to schemes with complicated singularities.  However, in the special case that the GIT quotient is also a locally symmetric space, the singularities can be expected to be simple.  Indeed, typically the discriminant will be identified locally analytically with the quotient of a hyperplane arrangement by a  finite group, and consequently the monodromy of the singularities of the schemes parameterized will be finite, forcing the singularities to be mild. In particular, the list of singularities occurring in the main GIT theorem (Theorem \ref{thmcubic}) can be explained (a posteriori) by the connection to the ball quotient model (Theorem \ref{mainthm}).

Moreover, in this situation, the discriminant can be blown-up in a standard way to obtain a simple normal crossings divisor, which can allow for the resolution of the period map to moduli spaces of abelian varieties or stable curves (see e.g. \cite{cml, ade}). In the case of this paper, as in \cite{cml}, an arithmetic hyperplane arrangement associated to the discriminant in Kondo's space allows for an explicit resolution of the period map $\overline M_4^{GIT} \dashrightarrow \overline M_4$; this is related to the more general process described in \cite{ade}.

Another motivation for analyzing the geometry of Kondo's ball quotient is the connection with the Hassett-Keel program, which aims to give a modular interpretation of the canonical model of $\overline M_g$.  This connection is discussed in section \ref{sectHK} (esp. Theorem \ref{thmhk}), and will be explored in more detail in a subsequent paper.  We note that the GIT  quotient of the Chow variety is expected to play an important role, in connection with a flip of the hyperelliptic locus.

\subsection*{Outline}
The main tool we use for the analysis of genus four curves is their close relationship with cubic threefolds. Specifically, a cubic threefold with an ordinary node determines a genus $4$ curve, and conversely.  We discuss this in detail in section  \ref{sectprelim}. Thus if $\overline M_{\text{cubic}}$ is a model for the moduli space of cubic threefolds, and $\Delta$ is the discriminant hypersurface (defined as the closure of the nodal locus), then the normalization $\Delta^\nu$ provides a  birational model for $\overline{M}_4$. In section \ref{sectgit} we review the results of Allcock \cite{allcock1} on the GIT quotient for the moduli space of cubic threefolds,  and we show that there is a natural GIT quotient for cubic threefolds with a  fixed singular point $(X,p)$. This quotient (via the construction in \S \ref{sectprelim}) gives a birational model $\overline{M}_{4}^{GIT}$ of the moduli of genus four curves. We then show that this quotient actually coincides with a GIT quotient of the Chow variety $\Chow_{4,1}$ of canonically embedded genus four curves (Theorem \ref{chowcubic}).  We note that in contrast to the cases of genus $3$ curves and cubic threefolds, there exist many natural choices for constructing a GIT quotient for genus $4$ curves. However, only one choice, the space $\overline{M}_4^{GIT}$, compares well with the ball quotient $(\calB_9/\Gamma)^*$.

In section \ref{sectstab} we describe the stability for the quotient $\overline{M}_{4}^{GIT}\cong \Chow_{4,1}\gquot \SL(4)$ (Theorem \ref{thmcubic}). Then, in section \ref{sectHK}, we identify $\overline{M}_{4}^{GIT}$ with a step in the Hassett--Keel program (Theorem \ref{thmhk}). We note that a partial analysis of the GIT on $\Chow_{4,1}$ was done by H. Kim \cite{kimg4} (motivated by the Hassett--Keel program).  The approach of \cite{kimg4}  is to directly compute GIT stability conditions for Chow varieties (vs. our approach via cubic threefolds); our results agree with those of Kim.  However, to our knowledge, Theorem \ref{thmcubic} is the first complete analysis for GIT stability on $\Chow_{4,1}$, and also the first description of the Hassett--Keel space $\overline M_4\left(\frac{5}{9}\right)$. We also point out a related GIT computation (also motivated by Hassett--Keel program): GIT for genus $4$ curves viewed as $(3,3)$ curves on a smooth quadric due to Fedorchuk \cite{maksymg4}.

In section \ref{sectball}, we recall the basic results of Kondo \cite{kondo}. In addition, we discuss some arithmetic results, e.g. the Baily--Borel compactification (Theorem \ref{propbb}), regarding the ball quotient model. The main result of the section is a computation of the canonical polarization of the ball quotient using Borcherds' automorphic form (Theorem \ref{thmballpol}).  In the final section,   we prove the main result comparing the GIT quotient to the ball quotient (Theorem \ref{mainthm}).  The proof uses the general framework due to Looijenga \cite{looijengacompact1} and the key point in this context is  the computation of the correct polarizations on the two spaces (Theorem \ref{thmballpol}). Finally, we note that both the GIT and ball quotient constructions  for genus $4$ curves are compatible with those for cubic threefolds (\cite{act}). Thus, our result essentially describes the  restriction to the discriminant of the \cite{act,ls} results.

\subsection*{Acknowledgements}
We would like to thank D. Allcock, S. Kondo, B. Hassett and E. Looijenga for discussions we have had on this topic. We are especially thankful to S. Kondo who shared with us some material related to the automorphic form computations in Section \ref{sectautomorphic}.

\subsection*{Notation and conventions}
We work over the complex numbers $\mathbb C$.  All \emph{schemes} will be taken to be of finite type over $\mathbb C$.
A \emph{curve} is a reduced, connected, complete scheme of pure dimension $1$.  We use the standard $ADE$ classification of simple singularities and we will say isolated  hypersurface singularities of different dimensions are of the same type if one is a stabilization of the other (see e.g. \cite{agzv1}).
We will use the notation $M_4^{nh}$ and  $M_4^{ns}$ to denote the open subsets  of $M_4$ parameterizing smooth non-hyperelliptic curves, and smooth (Brill-Noether) non-special curves (i.e. without vanishing theta null), respectively.  $\overline{M}_4^{GIT}$ is the GIT compactification of canonical curves constructed in \S 2, and $(B_9/\Gamma)^*$ is the Baily-Borel compactification of Kondo's ball quotient model. $\Sigma$ and $V$ will denote the divisors of nodal curves and curves with vanishing theta null respectively in the $\overline{M}_4^{GIT}$ model. By abuse of notation, we will sometimes use $\Sigma$ and $V$ to denote analogous divisors on related spaces.

\section{Preliminaries on canonical genus $4$ curves and cubic $3$-folds}\label{sectprelim}
In this paper, we will be interested in a GIT quotient of the space of canonically embedded, non-hyperelliptic, genus four curves.  Such curves are the complete intersection of a quadric and cubic in $\mathbb P^3$. Although these complete intersections can be parameterized naturally by a subset of the Hilbert scheme, or Chow variety, we find it is more convenient to work with the closely related projective bundle $\mathbb PE$ parameterizing subschemes of $\mathbb P^3$ with ideal defined by a quadric and cubic (\S \ref{sectpe}; for the relation to the Hilbert scheme see \cite{rv}).  The GIT quotient we consider is induced by a GIT problem for cubic threefolds.  In \S \ref{sectsing} we review the connection between genus four canonical curves and singular cubic threefolds. Finally, in  \S\ref{sectcubicg4},  we discuss associated maps among the spaces introduced.

\subsection{Complete intersections and genus four curves}\label{sectpe}
A smooth, genus $4$, non-hyperelliptic  curve is the complete intersection  of a quadric and a cubic in $\bP^3$.  We will call a scheme (resp. complete intersection) defined by a quadric and a cubic in $\mathbb P^3$ a \emph{$(2,3)$-scheme} (resp. \emph{$(2,3)$-complete intersection}).    The parameter space for $(2,3)$-schemes is a projective bundle
$$\pi:\bP E\to \bP H^0(\bP^3,\calO_{\bP^3}(2))\cong \bP^9$$ over the space of quadrics in $\bP^3$.   In this section, we discuss the geometry of $\mathbb PE$.

The vector bundle $E$, defining $\bP E$, can be constructed in the following way.  A $(2,3)$-scheme $C$ is defined by a quadric $Q$, say given by the equation $q$, and a cubic equation $f$ defined modulo $q$. Thus, the fiber of the bundle $E$ over  a point  $[q]\in \bP^9$ will be given by the exact sequence
$$0\to H^0(\mathbb P^3,\mathscr O_{\mathbb P^3}(1))\xrightarrow{q} H^0(\mathbb P^3,\mathscr O_{\mathbb P^3}(3))\to E_q\to 0.$$
Globally, one can define the bundle $E$ via the following exact sequence on  $\mathbb P^9$,
\begin{equation}\label{defpe}
0\to \pi_{2*}\left(\mathcal{I}_{\mathcal Q}\otimes \pi_1^*\calO_{\mathbb P^3}(3)\right)\to  \pi_{2*}\left(\pi_1^*\calO_{\mathbb P^3}(3)\right)\to E\to 0,
\end{equation}
where ${\mathcal Q}\subset \bP^3\times \bP^9$ is the universal quadric, and $\pi_1$ and $\pi_2$ are the natural projections onto $\bP^3$ and $\bP^9$ respectively.

The cohomology of projective bundles is well understood. Namely, $H^*(\mathbb PE,\bZ)$ is a free module over $H^*(\bP^9,\bZ)$ with basis $1,h,\dots, h^{15}$, where $h=\calO_{\bP E}(1)$. Also, $\Pic(\mathbb PE)$ is a free, rank two $\mathbb Z$-module generated by $h$ and $\eta$, where $\eta=\pi^*\calO_{\bP^9}(1)$. We will denote by $\calO(a,b)$ the line bundle on $\bP E$ of class $a\eta+bh$.

There are some geometric subloci of $\mathbb PE$ that will be of interest.  A smooth genus four curve is said to have a vanishing theta null if its canonical model lies on a quadric cone in $\mathbb P^3$.  We set $V\subseteq \mathbb PE$, {\it the vanishing theta null locus}, to be the locus of $(2,3)$-schemes whose defining quadric is singular.   Note that this is the pull-back from $\mathbb P^9$ of the discriminant for quadric hypersurfaces in $\mathbb P^3$, and consequently, $V$ is irreducible with generic point corresponding to a $(2,3)$-complete intersection lying on a quadric cone.  We set $\Sigma\subseteq \mathbb PE$ to be {\it the discriminant locus}; that is the locus of singular $(2,3)$-schemes.    This is a divisor, and
the generic point of $\Sigma$ corresponds to a $(2,3)$-complete intersection with a unique singularity, which is a node.
In particular, the locus of $(2,3)$-schemes with a singularity worse than a node is of codimension at least two (in fact the locus of curves with $A_2$ singularities is of codimension two).

Let $\mathscr C\to \mathbb PE$ be the universal $(2,3)$-scheme over $\mathbb PE$.
Let $\mathscr C^\circ\to \mathbb PE^\circ$ be the family of stable curves; the observation above shows that the
complement of $\mathbb PE^\circ$ in $\mathbb PE$ has codimension two.  There is an induced morphism $\mathbb PE^\circ \to \overline M_4$; the pull-back of the $\lambda$ and $\delta$ class then can be extended to $\mathbb PE$ over the codimension two locus.  Slightly abusing notation, we will denote these classes on $\mathbb PE$ again by $\lambda$ and $\delta$.  Note that $\delta$ agrees with $\Sigma$.

\begin{proposition}\label{computeclasses}
Let $\eta=\pi^*\mathscr O_{\mathbb P^9}(1)$ and $h=\mathscr O_{\mathbb PE}(1)$ be the standard generators of $\Pic(\bP E)$. Then $K_{\bP E}=-14\eta-16h$ and
\begin{eqnarray*}
V&=&4\eta\\
\Sigma&=&33\eta+34h.\\
\end{eqnarray*}
We also have,
\begin{eqnarray*}
\lambda &=& 4\eta+4 h\\
\delta &=&33\eta+ 34 h,
\end{eqnarray*}
and conversely $\eta=(17/2)\lambda-\delta$ and $h=-(33/4)\lambda+\delta$.
\end{proposition}

\begin{proof}
The computation of the canonical class of a projective bundle is standard.  The locus $V$ is the pull-back of the discriminant for quadric surfaces in $\mathbb P^3$, which has degree four.  The remaining classes can be computed with test curves.  For instance, one can fix a general quadric surface, and consider a general pencil of   cubics.  Or alternatively, one can fix a general cubic surface, and consider a general pencil of quadrics.  The classes of these two test curves are dual to the classes of $h$ and $\eta$ respectively.  Computing $\lambda$ and $\delta$ on these test curves is standard.  See for instance Harris--Morrison \cite[p.170-171]{hm} for the $\lambda$ class. The $\delta$ class can be computed easily, using for instance the standard method of topological Euler characteristics for Lefschetz pencils of curves on a smooth surface.
\end{proof}

\subsection{Cubic threefolds and genus four curves}\label{sectsing}
We begin by reviewing the following well known construction. Given a hypersurface $X\subset \mathbb P^n$ of degree $d$ with a singularity of multiplicity $d-1$ at the point $p=(1,0,\ldots,0)$, an equation for $X$ can be written as
$$
x_0q(x_1,\ldots,x_n)+f(x_1,\ldots,x_n)
$$
with $q$ and $f$ homogeneous of degrees $d-1$ and $d$ respectively. The ideal $(q,f)$ defines a scheme $Y\subseteq \mathbb P^{n-1}$ of type $(d-1,d)$. Conversely, given a complete intersection $Y\subseteq \mathbb P^{n-1}$ of type $(d-1,d)$ together with  a choice of generators $q$ and $f$ of the defining ideal, there is a hypersurface $X\subset \mathbb P^n$ of degree $d$ with a singularity of multiplicity $d-1$ at the point $(1,0,\ldots,0)$ defined by the equation  $x_0q+f$. In particular, in the case of cubic threefolds, setting $D_0\subseteq \mathbb PH^0(\mathbb P^4,\mathscr O_{\mathbb P^4}(3))$ to be the subset of cubic polynomials that are singular at $p$,
 we get an isomorphism
\begin{equation}\label{map1}
D_0\cong V_2 \times V_3,
\end{equation}
where $V_i=H^0(\mathbb P^3,\mathscr O_{\mathbb P^3}(i))$ for $i=2,3$
denotes the vector spaces of quadrics and cubics in $\mathbb P^3$ respectively.

It is also convenient to have the following more coordinate free description of this relationship.  Given a hypersurface $X$ as above, consider the projection from $p=(1,0,\ldots,0)$; this gives a dominant birational map $$\pi_p:X\dashrightarrow \mathbb P^{n-1}=V(x_0)\subset \mathbb P^n$$ given by $(x_0:\ldots:x_n)\mapsto (0:x_1:\ldots:x_n)$.  Since the projection $\mathbb P^n\dashrightarrow \mathbb P^{n-1}$ is resolved by blowing up the point $p$, the same is true for the map from $X$.  Clearly the exceptional locus of the map $\widetilde X:=\operatorname{Bl}_pX\to \mathbb P^{n-1}$ is the proper transform of the lines lying on $X$ passing through $p$ and one can check  that this is the cone over $Y=V(q,f)\subseteq V(x_0)=\mathbb P^{n-1}$ (see e.g \cite[Lem. 6.5]{cg} or \cite[Lem. 1.5]{ak77}).

We now  recall the well known connection between the singularities of $X$ and the singularities of $Y$.  First observe that if $p'\in X$ is a singular point other than $p$, then since $\operatorname{mult}_pX=d-1$, it follows that the line $\overline {pp'}$ is contained in $X$.  Thus for every singular point $p'\ne p\in X$, we have $\pi_p(p')\in Y$.
Now fix $y\in Y=V(q)\cap V(f)\subseteq V(x_0)=\mathbb P^{n-1}$.
The following are well known, and elementary to check (e.g. \cite{wallsextic}):
\begin{itemize}
\item[i)] If $Y$ is smooth at $y$, $X$ is smooth along the line $\overline{py}$  except at $p$.
\item[ii)] If $Y$ has a singularity at $y$ and $V(q)$ is smooth at $y$, $X$ has exactly two singular points $p,p'$ on the line $\overline{py}$.  Moreover, if $Y$ has a singularity of type $T$ at $y$, the singularity of $X$ at $p'$ has type $T$.
\item[iii)] If $Y$ has a  singularity  at $y$, $V(q)$ is singular at $y$,  and $V(f)$ is smooth at $y$, the only singularity of  $X$ along $\overline{py}$ is at $p$.  Moreover, if $Y$ has a singularity of type $T$ at $y$, $\textnormal{Bl}_pX$ has a singularity at of type $T$ at $y$ (where we have identified the exceptional divisor with $V(x_0)$).
\item[iv)] If $V(q)$ and $V(f)$ are both singular at $y$, $X$ is singular along $\overline{py}$. \end{itemize}
It follows that
\emph{if $X$ has only isolated singularities, then the  singularities of $\operatorname{Bl}_pX$  are in one-to-one correspondence, including the type, with the singularities of $Y$}.

\begin{remark} If $y\in V(q,f)$ is a singular point of a complete intersection, then $y$ is a hypersurface singularity  if and only if $V(q)$ and $V(f)$ are not both  singular at $y$; thus the comparison of types above is well defined using the stabilization of singularities. \end{remark}

 Using the classification of singularities (esp. \cite[\S15, \S16]{agzv1}), it is possible to make stronger statements in our situation. Namely, we have the following consequences for cubic threefolds with mild singularities, established in \cite{cml}.  Recall our convention that a curve is reduced, but possibly irreducible.

\begin{proposition}[{\cite[\S3]{cml}}] \label{comparesing}
Let $q(x_1,\ldots,x_4)$ (resp. $f(x_1,\ldots,x_4)$) be a homogeneous quadric (resp. cubic) polynomial on $\mathbb P^3$.
Set $X=V(x_0q+f)\subseteq \mathbb P^4$ and $C=V(q,f)\subseteq \mathbb P^3$.  Then $X$ has isolated singularities if and only if $C$ is a   curve with at worst hypersurface singularities.   Assuming either of these equivalent conditions hold:

\begin{enumerate}
\item The singularities of $\operatorname{Bl}_pX$, the blow-up of $X$ at $p=(1,0,0,0,0)$, are in one-to-one correspondence with the singularities of $C$, including the type. Note that if $p$ is a singularity of type $A_k$ , then $\operatorname{Bl}_pX$ has a unique singular point along the exceptional divisor, which is of type $A_{k-2}$ (smooth for $k\le 2$). Similarly, if $p$ is of type $D_4$ there are exactly three singular points of type $A_1$ along the exceptional divisor.

\item  The singularities of $X$ are at worst of type  $A_k$, $k\in \mathbb N$, $D_4$ if and only if the singularities of $C$ are at worst of  type $A_k$, $k\in \mathbb N$, $D_4$ and either $Q$ is irreducible, or $Q$ is the union of two distinct planes and $C$ meets the singular line of $Q$ in three distinct points. Moreover, under either of these equivalent conditions, the  singularity of $X$ at $p$ is of type:
\begin{enumerate}
\item $A_1$ if and only if $Q$ is a smooth quadric;
\item $A_2$ if and only if $Q$ is a quadric cone and $C$ does not pass through the vertex;
\item $A_k$ ($k\ge 3$) if and only if $Q$ is a quadric cone, $C$ passes through the vertex $v$, and the singularity of $C$ at $v$ is of type $A_{k-2}$;
\item $D_4$ if and only if $Q$ is the union of two distinct planes and $C$ meets the singular line of $Q$ in three distinct points.
\end{enumerate}
\end{enumerate}
\end{proposition}

In addition to cubics with isolated singularities, we need to consider the so-called chordal cubic threefolds. Namely, we recall that the secant variety of a rational normal curve in $\mathbb P^4$ is a cubic hypersurface, which is singular exactly along the rational normal curve; we will call this a {\it chordal cubic (threefold)}.  Occasionally, we will need to fix a specific chordal cubic.  We set {\it the standard rational normal curve} in $\mathbb P^n$ to be the one given by the map $(t:s)\mapsto (t^n,t^{n-1}s,\ldots,s^n)$.
The secant variety to the standard rational normal curve in $\mathbb P^4$ is called {\it the standard chordal cubic (threefold)}; note that the singular locus contains the point $p=(1,0,0,0,0)$. The following is easily established:

\begin{lemma}\label{lemchordal}
If $X$ is the standard chordal cubic, then the associated $(2,3)$-scheme is
$$C=V\left(x_2x_4-x_3^2,x_1(x_1x_4-x_2x_3)-x_2(x_1x_3-x_2^2)\right)\subseteq \mathbb P^3,$$
and the support of $C$ is the standard rational normal curve in $\mathbb P^3$.   Conversely, given a $(2,3)$-scheme  in $\mathbb P^3$ with support equal to a rational normal curve, the associated cubic is a chordal cubic. \qed
\end{lemma}

\begin{proof}
We provide a brief sketch, and  leave the details to the reader.  The equations for the $(2,3)$-scheme associated to the standard chordal cubic are easily worked out from its determinantal description.  One can check directly that the support is the standard rational normal curve of degree three.   Conversely, given a $(2,3)$-scheme in $\mathbb P^3$ with support equal to a rational normal curve of degree three, one uses i)-iv) above to show that the singular locus of the associated cubic $X$ contains a rational normal curve of degree four.  One concludes (from the paragraph before i)-iv)) that $X$ contains a chordal cubic, finishing the proof.
\end{proof}

\subsection{Rational maps to moduli spaces of curves}\label{sectcubicg4}
Let $\mathscr H_{cub} \cong \mathbb P^{34}$ be the Hilbert scheme of cubics in $\mathbb P^4$ and $\Delta\subset \mathscr H_{cub}$ the discriminant.
We define $\Delta_0\subset \Delta$ to be the locus of cubics that are singular at the point $p=(1,0,0,0,0)\in \mathbb P^4$. Clearly $\Delta_0$ is a projective space and $\Delta_0=\mathbb PD_0$ (see \S \ref{sectsing}).
The isomorphism \eqref{map1} induces a rational map
$$
\Delta_0\dashrightarrow \mathbb PV_2\times \mathbb PV_3.
$$
Composing with the rational map $\mathbb  P V_2\times  \mathbb PV_3\dashrightarrow \mathbb PE$ gives
\begin{equation}\label{map2}
\Delta_0\dashrightarrow \mathbb PE.
\end{equation}
This is regular outside of the locus of cubics that are reducible or have a triple point at $p$; indeed the map is given by the rule $[x_0q+f]\mapsto ([q],[\bar f])$, with  $\bar f(= f\mod q)$ , which is defined so long as $q$ is non-zero and $f$ is not divisible by $q$.

We have seen that there is a morphism $\mathbb PE^\circ\to \overline M_4$, where $\mathbb PE^\circ$ is the locus of $(2,3)$-schemes with  at worst nodal singularities.    Let $\Delta_0^{A_1}$ be the locus of cubics, singular at $p$, which have only nodes as singularities.
From Proposition \ref{comparesing} it follows that there is a morphism
$$
\Delta_0^{A_1}\to \mathbb PE^\circ\to \overline M_4.
$$
Let $G'$ be the subgroup of $\SL(5)$ that fixes the point $p$.  This group, and its action on $\Delta_0$, will be investigated in more detail in section \ref{sectgit} below. For now, we note that it is elementary to check that the morphism $\Delta_0^{A_1}\to \overline M_4$ is $G'$-invariant. We will consider the GIT problem later, but for now, we can conclude that there is a map of sets
$$\Delta_0^{A_1}/G'\to \overline M_4.$$

We now consider the map in the opposite direction.  Given a smooth genus four curve $C$, the canonical model is a $(2,3)$-complete intersection in $\mathbb P^3$, where $\mathbb P^3$ has been identified with $\mathbb PH^0(C,K_C)^\vee$.  The curve $C$ lies on a unique quadric defined by say $q$, and on a cubic $f$, unique up to linear multiples of $q$.  Thus we get exactly the data of a $G'$ orbit of a point in $\Delta_0$.

We expand this construction to families. Let $\overline M_4^{(2,3)}$ be the locus of curves $C$ such that the canonical model $\phi(C)$ satisfies  $h^0(\mathbb P^3,I_{\phi(C)}(2))=1$ and $h^0(\mathbb P^3,I_{\phi(C)}(3))=5$; let $\overline {\mathcal M}_4^{(2,3)}$ be the associated sub-stack of $\overline {\mathcal M}_4$.
Let $g:\mathscr C\to B$ be a family in $\overline {\mathcal M}_4^{(2,3)}$.
  Let $B'\to B$ be an \'etale cover such that for the induced family $g':\mathscr C'\to B'$, the bundle $g'_*\omega_{\mathscr C'/B'}$ is trivialized.  The relative canonical embedding can then be viewed as a family of curves in $\mathbb P^3$.  The family being an object of $\overline {\mathcal M}_4^{(2,3)}$ implies that this is a family of $(2,3)$-complete intersections in $\mathbb P^3$.  Moreover, the family of quadrics is well-defined up to the action of $\SL(4)$, and the family of cubics is well defined up to the linear combinations of the quadric.  In other words, we get a well defined map of sets $B\to \Delta_0/G'$.
This induces a map of the underlying sets $\overline {M}_4^{(2,3)}\to \Delta_0/G'$.  Once we establish a GIT quotient (scheme) for the target in the subsequent sections, we will be able to conclude that this induces a morphism
$$\overline {M}_4^{(2,3)}\to \Delta_0\gquot G'.$$

Note that a theorem of Rosenlicht (see also  \cite[Thm. 4.3]{km}) states that an integral curve of genus $g$ and degree $2g-2$ in $\mathbb P^{g-1}$ is  non-hyperelliptic, Gorenstein, and is its own canonical model.  A theorem due to Fujita (see also  \cite[Prop. 5.5]{km}) states that such a curve is projectively normal, in the sense that for every $n\ge 1$, the hypersurfaces of degree $n$ cut out a complete linear system.  In particular, we  conclude that the generic points of the boundary divisors $\delta_0,\delta_1$ in $\overline{M}_4$ are contained in $\overline {M}_4^{(2,3)}$.

\begin{example}\label{exaC2A5}  Let $B\subseteq \mathbb C$ be the unit disc.
Let $\mathscr C\to B$ be a generic deformation of a generic curve $C\in \delta_2$.  This defines a morphism $B\to \overline M_4$ passing through the generic point of $\delta_2$, and via the construction above, a map $B^\circ =B-B\setminus\{0\}\to \Delta_0\gquot G'$.  We want to describe an extension of this morphism over the central point.
To do this, label the irreducible components of the central fiber $C$ as $C_1$ and $C_2$, and let us say they are attached at the points $p_1$ and $p_2$ respectively.  Blow-up $\mathscr C$ at the points which are the hyperelliptic involutions of $p_1$ and $p_2$ on the respective curves, as well as at $p_1=p_2$.  The result is a family with central fiber a chain of five curves: $C_1$, $C_2$, and three rational curves.
Twisting the relative dualizing sheaf by the appropriate divisors on the total space of the family supported on $C$, gives a line bundle which is degree one on each of the two rational tails, degree four on the rational bridge, and degree zero on the genus two curves.  The morphism associated to the line bundle gives a family of smooth $(2,3)$ curves degenerating to a curve which consists of two lines meeting a degree four smooth rational curve in distinct points; the singularity type of each intersection is type $A_5$ (i.e. the singularity type of the central fiber is $2A_5$).   The associated family of cubics, say $\mathscr X\to B$ has generic fiber equal to a cubic with a unique singularity, which is a singularity of type $A_1$ at $p=(1,0,0,0,0)$.  The central fiber has exactly three singularities, one of type $A_1$ at $p$, and two $A_5$ singularities.
\end{example}

\section{GIT for $g=4$ curves via cubic $3$-folds}\label{sectgit}
In this section we discuss a compact GIT model $\overline{M}_4^{GIT}$ for canonical genus $4$ curves induced from the GIT quotient for the moduli space of cubic $3$-folds. We then consider the projective bundle $\bP E$, discussed in \S\ref{sectpe}, parameterizing $(2,3)$-schemes (in $\mathbb P^3$), and show that $\overline{M}_4^{GIT}\cong \bP E \gquot_{\calO(3,2)} \SL(4)$ (Proposition \ref{GITIdentification}).  Finally, we identify $\overline{M}_4^{GIT}$ with a more standard GIT quotient, the GIT quotient of the Chow variety $\Chow_{4,1}$ associated to  genus $4$ curves.

\subsection{GIT for cubic $3$-folds}
We start by reviewing  the results of Allcock \cite{allcock1} on the GIT quotient for cubic threefolds.  As usual, change of coordinates gives an action of $G=\SL(5,\mathbb C)$ on $\mathscr H_{cub} \cong \mathbb P^{34}$ and there is a GIT quotient $\mathscr H_{cub}\gquot G$.
To describe the GIT stability of cubic threefolds, Allcock \cite[p.2]{allcock1} considers the family of cubics
$$
F_{A,B}=x_0(x_3^2-x_2x_4)+Ax_2^3+Bx_1x_2x_3+x_1^2x_4,
$$
(for $(A,B)\in \mathbb C^2\setminus \{(0,0)\}$) and the cubic threefold (not of type $F_{A,B}$) defined by
$$
F_D=x_0x_1x_2+x_3^3+x_4^3.
$$
In particular, the case $A=1,B=2$ gives the chordal cubic discussed above:
$$
F_c=-\det
\left( \begin{array}{ccc}
x_0 & x_1 & x_2 \\
x_1 & x_2 & x_3 \\
x_2 & x_3 & x_4 \end{array} \right).
$$
Note that these cubics are singular at the point $p=(1,0,\dots,0)$; thus they belong to the locus $\Delta_0\subset \mathscr H_{cub}$.  Note also that  the involution $\tau\in \SL(5,\mathbb C)$ determined by $x_i\mapsto x_{4-i}$ has the property that $\tau F_{A,B}=F_{A,B}$.  Consequently every cubic in $F_{A,B}$ is singular at $p':=\tau(p)=(0,0,0,0,1)$ as well, and if the singularity at $p$ is isolated, then the singularities at $p$ and $p'$ are of the same type.  One can also check that $F_{0,B}$ has an $A_1$ singularity at the point $(0,0,1,0,0)$.

Let $X_1$ and $X_2$ be cubic threefolds.  We say that
$X_1$ degenerates to $X_2$ if the latter is in the closure of the $G$-orbit of the former.  For a hypersurface in $\mathbb P^n$, the quadratic terms of a singularity define a quadric form on the tangent space to $\mathbb P^n$.   The kernel of this form determines a linear subspace of $\mathbb P^n$, called the null space of the singularity; the dimension of this space is called the nullity (and also the corank).  The nullity of an $A_n$ ($n>1$) singularity is one, and the nullity of a $D_4$ singularity is two.
We also note the following.   Let $X$ be a cubic threefold with a double point $x\in X$.  Let $\pi_x:\mathbb P^4\dashrightarrow \mathbb P^3$ be the projection from $x$.  Let $Q_x\subseteq \mathbb P^3$ be the quadric determined by $x$.  The null space associated to $x$ can be identified with $\pi_x^{-1}(\operatorname{Sing}Q_x)$.

 Now, we can state the GIT analysis for cubic threefolds as follows.

\begin{theorem}[{Allcock \cite{allcock1}}]\label{thmallcock} Let $X\in \mathscr H_{cub}$ be a cubic threefold.    The GIT stability of $X$ with respect to the natural linearization is described as follows.
 \begin{enumerate}
\item $X$ is stable if and only if it has at worst $A_1,\ldots,A_4$ singularities.
\item The minimal orbits of strictly semi-stable cubic threefolds are the orbits of $F_D$ and of the $F_{A,B}$; i.e. these are the poly-stable cubics.

\item $X$ is strictly semi-stable if and only if \begin{enumerate}
\item $X$ contains a $D_4$ singularity, in which case $X$ degenerates to $F_D$, or,
\item $X$ contains an $A_5$ singularity, in which case $X$ degenerates to $F_{A,B}$ for some $A,B$ such that $4A/B^2\ne 1$, or,
\item $X$ contains an $A_n$ singularity ($n\ge 6$), but does not contain any of the planes containing its null line, in which case $X$ degenerates to $F_c$, or,
\item $X$ is a chordal cubic.
\end{enumerate}
\item $X$ is unstable if and only if \begin{enumerate}
\item $X$ has non-isolated singularities and is not a chordal cubic, or,
\item $X$ contains an isolated singularity other than an $A_1,\ldots,A_5$ or $D_4$ singularity, and, if this singularity is of type $A_n$ ($n\ge 6$) then $X$ contains a plane containing its null line.
\end{enumerate}
\end{enumerate}
\end{theorem}

\begin{remark} The following,  shown in \cite{allcock1}, clarifies which cubics are parameterized by $F_{A,B}$ and $F_D$.
 $V(F_{A,B})\cong V(F_{A',B'})$ if and only if $4A/B^2=4A^{'}/B^{'2}\in \mathbb C\cup \infty$.  If $4A/B^2 \ne 0,1$, then $F_{A,B}$ has exactly two singularities, both of type $A_5$.  If $4A/B^2=0$, then $F_{A,B}$ has exactly three singularities, two of type $A_5$ and one of type $A_1$.  If $4A/B^2=1$, $F_{A,B}$ is a chordal cubic.  $F_D$ has exactly three singularities, each of type $D_4$.
\end{remark}

\subsection{The definition of the space  $\overline M_4^{GIT}$} As discussed in \S \ref{sectprelim} (esp. \S\ref{sectcubicg4}), there is close relationship between (the normalization of) the discriminant divisor in the moduli space of cubic threefolds and the moduli of canonically embedded genus $4$ curves. Here we will interpret Theorem \ref{thmallcock} as a GIT result for canonically embedded, genus $4$ curves.

We start by defining a space $\overline M_4^{GIT}$ as the normalization of the discriminant divisor for cubic $3$-folds:
\begin{equation}
\overline M_4^{GIT}:=\left(\Delta\gquot G\right)^\nu\to \Delta\gquot G \subset \calH_{cub}\gquot G,
\end{equation}
where $G=\SL(5)$, $\Delta$ is the discriminant hypersurface in the parameter space for cubics $\calH_{cub}=\bP H^0(\bP^4,\calO_{\bP^4}(3))$, and the superscript $\nu$ denotes the normalization. The notation is justified by the fact  $\overline M_4^{GIT}$ is a projective variety (by construction) which is birational to the moduli of genus $4$ curves $\overline{M}_4$ (see \S \ref{sectcubicg4}).   We also point out the following.  Let $\nu:\Delta^\nu\to \Delta$ be the normalization, and let $L^\nu$ be the pull-back of the linearization on $\Delta$.  $L^\nu$ is ample, and it is well known that $(\Delta\gquot G)^\nu=\Delta^\nu\gquot G$.    Moreover, since $\Delta^\nu$ is complete, and normalization maps are finite, it follows from \cite[Thm. 1.19]{git} that a point $x\in \Delta^\nu$ is stable (resp. semi-stable, poly-stable) if and only if $\nu(x)$ is stable (resp. semi-stable, poly-stable).   Thus Allcock's results give a complete description of the stability conditions on $\Delta^\nu$ as well.

The difficulty in immediately identifying $\overline M_4^{GIT}$ with a moduli space of curves is that in obtaining a curve in $\mathbb P^3$ from a singular cubic threefold, one must choose a singular point, as well as a choice of coordinates for the projection.  The former ambiguity is essentially taken care of by the normalization, but the latter still remains an issue.  In other words, there is not a family of curves lying over $\Delta^\nu$.  However, as we discussed in \S\ref{sectcubicg4}, there is a family of curves  lying over the related space $\Delta_0$, the locus of cubics singular at the fixed point $p$.  Thus, our first step is to describe $\overline M_4^{GIT}$  as a quotient of $\Delta_0$ (instead $\Delta$).

\subsection{$\overline M_4^{GIT}$ as a non-reductive quotient}\label{nrquotient} 
As explained above, we are interested in describing $\overline M_4^{GIT}$ as a quotient of $\Delta_0$. The obvious choice of quotient is  $\Delta_0\gquot G'$, where  $G'\subset G$ is the parabolic subgroup stabilizing $p\in \bP^3$. Since the group $G'$ is not reductive, the main issue is to make precise the meaning of the quotient of $\Delta_0$ by $G'$, and to prove that such a quotient exists. 

To start, we define $\Delta_0\gquot G'$ as the $\Proj$ of the ring $R'$ of $G'$-invariant sections of powers of the polarization $\calO(1)$ on $\Delta_0$ (N.B. $\Delta_0\cong \bP^{29}$).  However, since  $G'$ is not reductive, finite generation of the ring $R'$ is not automatic. Following Kirwan \cite[\S3]{kirwannr} (see also \cite{dkir}),  we handle this issue by replacing the action of the non-reductive group $G'$ by the action of a reductive group $G$ (containing $G'$) on a related quasi-projective variety. As before, we take $G=\SL(5)$ and consider the 
variety $\widetilde{\Delta}=G\times_{G'} \Delta_0$, where as usual $G\times_{G'} \Delta_0$ is the quotient of $G\times \Delta_0$ by the free action of $G'$: $h(g,X)=(gh^{-1},hX)$ (for $h\in G'$). In our situation $G/G'\cong \bP^4$ and it is not hard to see that $\widetilde \Delta$ coincides with the space of cubics with a marked singularity:
$$\widetilde \Delta =\left\{ (X,x)\mid X \textrm{ is a cubic threefold singular at } x\right\}\subset \calH_{cub}\times \bP^4.$$
It is  well known that $\widetilde \Delta$ is determinantal (it is a Fitting scheme associated to a map of cotangent bundles), of the expected dimension, and normal.

\begin{notation}
Let $X$ be a projective variety and $L$ a (not necessarily ample) line bundle. We denote $R(X,L):=\oplus_{n\ge 0} H^0(X,L^{\otimes n})$ the ring of sections of $L$. If a group $H$ (not necessarily reductive) acts on $L$, we denote $R(X,L)^H\subseteq R(X,L)$ the subring of $H$-invariant sections.  If $R(X,L)^H$ is finitely generated, then we define $X\gquot_L H:=\Proj R(X,L)^H$. If $L$ is ample and $H$ reductive, $X\gquot_L H$ is the standard GIT quotient.
\end{notation}

The pull back of the line bundle $\pi_2^*\mathscr O_{\mathbb P^{34}}(1)$ to $\widetilde \Delta$ gives a  line bundle $\tilde L$ on $\widetilde \Delta$.
Note also that the natural action of $G$ on $\calH_{cub}\cong \mathbb P^{34}$ extends the action of $G'$, so that
$$
G\times _{G'}\mathbb P^{34}\cong (G/G')\times \mathbb P^{34}=\mathbb P^4\times \mathbb P^{34},
$$
where the isomorphism on the left is given by the rule $[g,x]\mapsto (gG',gx)$.
In particular, the line bundle obtained by pulling back $\pi_2^*\mathscr O_{\mathbb P^{34}}(1)$ to $\widetilde \Delta$ via the embedding $$\widetilde \Delta=G\times _{G'}\Delta_0\subseteq G\times_{G'}\mathbb P^{34},$$ (analogous to the line bundle considered in \cite[p.10]{kirwannr})  is equal to  $\tilde L$.
In addition, and again similar to the case studied in \cite[p.10]{kirwannr}, there is a natural identification of the ring of invariants:
\begin{equation}\label{invring}
R':=R(\Delta_0,L)^{G'}\cong R(\widetilde{\Delta},\tilde L)^{G},
\end{equation}
where $L=\mathscr O(1)$ on $\Delta_0$ is the natural polarization induced from the inclusion
$\Delta_0\subseteq \mathbb P^{34}$. In other words, we have replaced a non-reductive GIT quotient $\Delta_0\gquot G'$ by a reductive GIT quotient $\widetilde{\Delta}\gquot G$, but the main issue, the finite generation of $R'$,  still remains:  we can not apply directly the standard GIT results (since $\tilde L$ is not ample). We solve this issue as follows:

\begin{proposition}\label{FiniteGeneration}
There is a morphism $\widetilde \Delta\to \Delta^\nu$ such that the pull-back of $L^\nu$ to $\widetilde \Delta$ is equal to $\tilde L$ and 
$$
R':=R(\Delta_0,L)^{G'}\cong R(\Delta^\nu, L^\nu)^{G}.
$$
In particular, $R'$ is finitely generated, and $\Delta_0\gquot G'$ is well defined and isomorphic to $\overline M_4^{GIT}$. 
\end{proposition}

\begin{proof}
Our geometric description shows that $\widetilde\Delta\subseteq \mathbb P^4\times \mathbb P^{34}$
admits a forgetful map to $\Delta$.
 $\tilde \Delta$ is normal, and
consequently this map  factors through the normalization $\widetilde \Delta \to \Delta^\nu\stackrel{\nu}{\to} \Delta$.     From the definitions it is clear that $\tilde L$ is the pull-back of $\mathscr O(1)$ from $\Delta$.  We set $L^\nu$ to be the pull-back of $\mathscr O(1)$ to $\Delta^\nu$.
Thus  the result will be proven  provided we show there is an isomorphism   $\oplus_n H^0(\widetilde \Delta,\tilde L^{\otimes n})^G\cong \oplus_n H^0(\Delta^\nu,(L^\nu)^{\otimes n})^G$.  But $\widetilde\Delta$ and $\Delta^\nu$ agree outside of the locus of cubics with positive dimensional singular locus, which is codimension at least two in both spaces.   Thus the spaces of sections agree, and the result is proven.
\end{proof}

\subsection{The  GIT quotient of the projective bundle}
Using Proposition \ref{FiniteGeneration} and the discussion of \S\ref{sectcubicg4}, we can now identify $\overline M_4^{GIT}$ with a standard   GIT quotient for genus $4$ curves.

\begin{proposition}
\label{GITIdentification}  Pulling back sections via the rational map $\Delta_0\dashrightarrow  \mathbb PE$ defines an isomorphism
$$
 R( \mathbb PE,\mathscr O(3,2))^{\SL(4)}\to R'=R(\Delta_0,L)^{G'}.
$$
Thus  $
\mathbb{P}E \gquot_{\mathcal{O} (3,2)} \SL (4)  \cong \overline M_4^{GIT}$.
\end{proposition}

\begin{proof}
Since we define our GIT quotients as $\Proj $ of rings of invariant sections, it is immediate to see that the following holds: {\it if $G$ is a group acting on a quasi-projective variety $X$, and $H$ is a normal subgroup, then $X\gquot G\cong (X\gquot H)\gquot (G/H)$}.   

In our situation, we have $G'\subset \SL(5)$ is the stabilizer  of a point and thus
$$G'= \left \{ \left(\begin{matrix} c& \vec{v}\\0& A \end{matrix}\right), \ A\in \SL(4),\ c=(\det A)^{-1}\in \bC^*,\ \vec{v}\in \bC^4\right\}.$$ 
Thus, we have the center $Z(G')\cong \bC^*$, and then $G'/Z(G')$ is a semidirect product $\bC^4 \rtimes \SL(4)$ (up to isogeny). From the discussion of the previous paragraph, it follows that we can understand $\Delta_0\gquot G'$ in three steps: first we quotient by the center $\bC^*$, then by the unipotent radical $\bC^4$, and finally by the reductive group $\SL(4)$.  

For the first step, we claim that there is a natural isomorphism
$$ \Delta_0 \gquot \mathbb{C}^* \cong \bP V_2 \times \bP V_3 $$
which identifies the line bundle $\mathcal{O} (1)$ on $\Delta_0$ with $\mathcal{O} (3,2)$ on $\bP V_2 \times \bP V_3$ (where $V_i=H^0(\bP^3,\calO_{\bP^3}(i))$). Note that the $\mathbb{C}^*$-action  on $\Delta_0\cong \bP^{29}$ is given by
$$ t \cdot (x_0 q + f) = t^{-2} x_0 q + t^3 f,$$
where $q$ and $f$ are homogeneous forms in $(x_1,\dots,x_4)$ of degree $2$ and $3$ respectively. The identification  $ \Delta_0 \gquot \mathbb{C}^* \cong \bP V_2 \times \bP V_3 $ then follows from a straightforward identification of the (semi-)stable locus (compare \cite[Ex. 2.5]{kirwannr}).
To see that the given line bundles are identified, we consider the pullback map
$$ H^0 ( \bP V_2 \times \bP V_3 , \mathcal{O} (a,b)) \to H^0 ( \Delta_0 , \mathcal{O} (a+b)) $$
and note that the image is invariant under the action of $\mathbb{C}^*$ if and only if $2a=3b$. 

For the second step, i.e. the quotient by the action of $\bC^4$, we note that the action is given explicitly by $(\alpha_1,\dots,\alpha_4)\in \bC^4$ acts on $(q,f)\in\bP V_2\times \bP V_3$ by
\begin{equation}\label{action}(\alpha_1,\dots,\alpha_4)\cdot (q,f)\to \left(q,f+q\cdot \sum \alpha_i x_i\right).\end{equation} 
Thus, the quotient of $\bP V_2\times \bP V_3$ by $\bC^4$ corresponds is the space of pairs $(q, f\mod q)$. We have already considered this space; it is $\bP E$ in the notation of \S\ref{sectpe}. In other words, the natural map  $\bP V_2 \times \bP V_3 \dashrightarrow \mathbb{P}E$ (which is regular as long as $q\not |f$)  is in fact the quotient map for the action of $\bC^4$. The choice of line bundle $\calO(a,b)$ is relevant for the choice of scaling factor for $q$ and $f$ in equation \eqref{action}.

The final step is the natural quotient by $\SL(4)$. We conclude, 
\begin{eqnarray*}
\Delta_0\gquot G'&\cong& \left((\Delta_0\gquot \bC^*)\gquot \bC^4\right))\gquot \SL(4)\\
&\cong&\left( \bP V_2\times \bP V_3\gquot_{\calO(3,2)} \bC^4\right)\gquot\SL(4)\\
&\cong& \bP(E)\gquot_{\calO(3,2)} \SL(4)
\end{eqnarray*}
as needed. We reiterate that all the isomorphisms above should be understood in the sense of rings of invariant sections.
\end{proof}

\subsection{The GIT quotient of the Chow variety}\label{sectchow}
A standard way of constructing models for moduli of curves is to consider 
GIT quotients of Chow varieties parameterizing (pluri)canonical curves (and their degenerations). For example, Mumford  \cite{mumford} constructed $\overline{M}_g$ as a projective variety in this way. Similarly, Schubert \cite{Schubert} obtained the pseudostable curve model $\overline{M}_g^{ps}$. More recently Hassett--Hyeon \cite{hh2} gave another model $\overline{M}_g^{cs}$ of $\overline{M}_g$ using appropriate quotients of Chow varieties. Here, we show that our model $\overline M_4^{GIT}$ is in fact the GIT quotient of the Chow variety $\Chow_{4,1}$ associated to canonical curves in $\bP^3$. We note that partial results on $\Chow_{4,1}\gquot \SL(4)$ were  obtained by H. Kim \cite{kimg4}.

We start our discussion by recalling some basic facts about quotients of Chow varieties (this is mostly based on \cite{mumford}). Let $X$ be a variety of dimension $r$ in $\bP^N$.  For a 1-parameter subgroup $\lambda$ of $\SL (N+1)$, we will write $x_i$ for homogeneous coordinates on $\bP^N$ that diagonalize $\lambda$.  Then there is a set of nonnegative integers $r_i$ such that $\lambda (t) x_i = t^{(N+1)r_i - \sum r_i} x_i$ (N.B. this differs from the other standard convention for the ``weights'' of a 1-PS).  Let $\alpha : \tilde{X} \to X$ be a proper birational morphism of varieties and $X' = \tilde{X} \times \mathbb{A}^1$.  Furthermore, let $\mathcal{I}$ be the ideal sheaf of $\calO_{X'}$ defined by
$$ \mathcal{I} \cdot [ \alpha^* \calO_X (1) \otimes \calO_{\mathbb{A}^1} ] = \text{ subsheaf generated by } t^{r_i} \alpha^* x_i $$
Next consider the function
$$ p(n) = \chi ( \calO_{X'} (n)/ \mathcal{I}^n \calO_{X'} (n) ) $$
For $n$ sufficiently large, $p(n)$ is a polynomial of degree $r+1$.  We write $e_{\lambda} (X)$ for the normalized leading coefficient of $f$, i.e.  the integer such that
$$ p(n) = e_{\lambda} (X) \frac{n^{r+1}}{(r+1)!} + \text{ lower order terms}. $$
We will use the following result due to Mumford \cite{mumford}:
{\it A Chow cycle $X$ is semistable if and only if
\begin{equation}\label{thmmum}
 e_{\lambda} (X) \leq \frac{r+1}{N+1} \deg (X) \sum r_i 
\end{equation}
for every one-parameter subgroup $\lambda$}.

In the case of genus $4$ canonical curves, we let $\Chow_{4,1}$ be the associated Chow variety and consider the natural GIT quotient $\Chow_{4,1}\gquot \SL(4)$. We first note:

\begin{proposition}
\label{ChowCompleteIntersection}
Every Chow semistable curve $C$ is the complete intersection of a quadric and a cubic.  If $C$ is Chow stable, then the quadric is irreducible.
\end{proposition}

\begin{proof}
If $C$ is not a complete intersection, then it is contained in a reducible quadric.  It follows that $C= C_1 + C_2$, where each $C_i$ is contained in a hyperplane $H_i$.  Without loss of generality, assume that $\deg ( C_1 ) \geq \deg ( C_2 )$.  Choose coordinates such that the hyperplane $H_1$ is cut out by $x_0$ and consider the 1-PS with weights $(0,1,1,1)$.  By  \cite[Lem. 1.2]{Schubert}, we know that
$$ e_{\lambda} (C) \geq 2 \deg ( C_1 ) + \deg (C_2 \cap H_1 ) = 6 + \deg ( C_1 ). $$
If $C$ is semistable,  we must have from \eqref{thmmum}
$$ 6 + \deg ( C_1 ) \leq 9.$$
It follows that $\deg ( C_1 ) = \deg ( C_2 ) = 3$, so $C$ is the union of two plane cubics.
\end{proof}

Note that there is a natural birational map 
$$\varphi: \mathbb PE \dashrightarrow \Chow_{4,1}$$ 
induced from the Hilbert-Chow morphism. We now can prove the main result of the section.

\begin{theorem}\label{chowcubic}
The pull-back of sections via  $\varphi$ induces an isomorphism
$$
R(\Chow_{4,1},\mathcal{O}_{\Chow_{4,1}} (1))^{\SL(4)}\to R(\mathbb PE,\mathscr O(3,2))^{\SL(4)}.
$$
Thus, $\Chow_{4,1} \gquot \SL (4)\cong \overline M_4^{GIT}$.
\end{theorem}

\begin{proof}
The  map $\varphi$ is regular along the open set $U \subset \mathbb PE$ of pairs $(q , f )$ where $q$ and $f$ do not share a common factor.  We first show that if $A = \mathcal{O}_{\Chow} (1)$, then the pullback of the ample class $\varphi^* A$ is linearly equivalent to a multiple of $\mathcal{O} (3,2)$.   Recall that Hassett-Hyeon have shown that $A$ corresponds to a multiple of $9\lambda -\delta$ at the level of $\overline M_4$ (\cite[Prop. 5.2]{hh2}).  It then follows from Proposition \ref{computeclasses} that $\varphi^*A=3\eta+2h$.  

Now, by Proposition \ref{ChowCompleteIntersection} we know that $\Chow_{4,1}^{ss} \subset \varphi(U)$, and we also observe that the complement of $U$ has codimension at least $2$ in $\mathbb PE$.  Hence the restriction maps
$$ H^0 ( \Chow_{4,1} , A^{\otimes n})^{\SL(4)} \tilde \to H^0 ( \varphi (U) , A^{\otimes n})^{\SL(4)} $$
$$ H^0 ( \mathbb PE , \mathcal{O} (3,2)^{\otimes n})^{\SL(4)} \tilde \to H^0 ( U , \varphi^* A^{\otimes n})^{\SL(4)} $$
are isomorphisms.    Since $ H^0 ( \varphi (U) , A^{\otimes n})^{\SL(4)} =H^0 ( U , \varphi^* A^{\otimes n})^{\SL(4)} $, the conclusion follows.
\end{proof}

\section{Stability for canonical genus $4$ curves}\label{sectstab}
From Allcock's theorem (Theorem \ref{thmallcock}) and the discussion of section \ref{sectprelim}, it is easy to describe the curves corresponding to the points of $\overline M_4^{GIT}$. Specifically, here we prove Theorem \ref{thmcubic}, which gives a complete description of the stability for the natural GIT quotient $\Chow_{4,1}\gquot \SL(4)\cong \overline M_4^{GIT}$. We note that our stability computation agrees with the partial analysis of Kim \cite{kimg4} (who makes a direct computation of the stability conditions on $\Chow_{4,1}$).

To state our result, we define the  family
$$
C_{A,B}=(x_3^2-x_2x_4, Ax_2^3+Bx_1x_2x_3+x_1^2x_4),
$$
(for $(A,B)\in \mathbb C^2\setminus \{(0,0)\}$) and the scheme (not of type $C_{A,B}$) defined by
$$
C_D=(x_1x_2,x_3^3+x_4^3),
$$
induced by the associated cubics considered by Allcock.
We also introduce the scheme (not of type $C_{A,B}$ or $C_D$) defined by
$$
C_{2A_5}=(x_1x_4-x_2x_3,x_1x_3^2+x_2^2x_4).
$$
This is the curve obtained from projecting $F_{0,1}$ from the $A_1$ singularity at the point $(0:0:1:0:0)$.
Note that it is elementary to check that each scheme  $C_{A,B}$ is singular at the points $q=(1:0:0:0)$ and  $q'=(0:0:0:1)$.  From the connection between singularities of cubic threefolds and the associated $(2,3)$-schemes, it follows that if $q,q'$ are isolated singularities, then the singularity at $q'$ is of type $A_5$ and the singularity at $q$ is of type $A_3$.

We now conclude with the following description of $\overline{M}_4^{GIT}$.  

\begin{theorem} \label{thmcubic} The stability conditions for the quotient $\Chow_{4,1}\gquot \SL(4)\cong \overline M_4^{GIT}$ are described as follows:
\begin{itemize}
\item[(0)] Every semistable point $c\in \Chow_{4,1}$ is the cycle associated to a $(2,3)$-complete intersection in $\mathbb P^3$. The only non-reduced $(2,3)$-complete intersections that give a semi-stable point $c\in \Chow_{4,1}$ are the genus $4$ ribbons (all with associated cycle equal to the twisted cubic with multiplicity $2$).  
 \end{itemize} 
Assume now  $C$ is a reduced  $(2,3)$-complete intersection in $\mathbb P^3$, with associated point $c\in \Chow_{4,1}$. Let $Q\subseteq \mathbb P^3$ be unique quadric containing $C$. Then the following hold:
 \begin{itemize}  
 \item[(0')] $c$ is unstable if $C$ is the intersection of a quadric and cubic that are simultaneous singular. Thus, in items (1) and (2) below we can assume $C$ has only hypersurface singularities. 
 \item[(1)] $c$ is stable if and only if $\rank Q\ge 3$ and $C$ is a curve with  at worst $A_1,\ldots,A_4$ singularities at the smooth points of $Q$ and at worst an $A_1$ or $A_2$ singularity at the vertex of $Q$ (if $\rank Q=3$). 
 \item[(2)] $c$ is strictly semi-stable if and only if
 \begin{itemize} \item[i)] $\rank Q=4$  and 
\begin{itemize} \item[($\alpha$)] 
$C$ contains a singularity of type $D_4$ or $A_5$, or,

\item[($\beta$)] 
$C$ contains a singularity of type $A_k$, $k\ge 6$, but does not contain such a singularity on a component $C'$ contained in a plane, or, 
 \end{itemize}

  \item[ii)] $\rank Q=3$,   $C$ has at worst an $A_k$, $k\in \mathbb N$, singularity at the vertex of $Q$  and   
  \begin{itemize}

  \item[($\alpha$)] $C$ contains a $D_4$ or an $A_5$ singularity at a smooth point of $Q$ or an $A_3$ singularity at the vertex of $Q$, or, 
  
  \item[($\beta$)] $C$ contains a singularity of type $A_k$, $k\ge 6$, at a smooth point of $Q$ or a singularity of type $A_k$, $k\ge 4$, at the vertex of $Q$, but does not contain such a singularity on a component $C'$ contained in a plane, or, 
  \end{itemize}
  \item[iii)] $\rank Q=2$ and $C$ meets the singular locus of $Q$ in three distinct points.  
\end{itemize} 
\end{itemize} 
Finally, the minimal orbits of the strictly semi-stable points are described as follows:
\begin{itemize}
  \item[(3)] The minimal orbits of strictly semi-stable $c$ are the orbits of the cycles associated to the $(2,3)$-subschemes given by  $C_{2A_5}$, $C_D$ and the $C_{A,B}$. 
  \end{itemize}
 In particular, the GIT boundary consists of $2$ isolated points (corresponding to $C_{2A_5}$ and $C_D$) and a rational curve $Z$(parametrizing the orbits of $C_{A,B}$). The orbit of the double twisted cubic corresponds to a special point of $Z$ (corresponding to $C_{A,B}$ with  $\frac{4A}{B^2}=1$).
\end{theorem}

\begin{proof}
The first part of item (0) is the content of Proposition \ref{ChowCompleteIntersection}. This allows us to restrict to the locus $U\subset \bP E$ corresponding to $1$-dimensional $(2,3)$-schemes (complete intersections). As before, we have a cycle map $\varphi: \bP E\dashrightarrow \Chow_{4,1}$, which is regular along $U$, and in fact an isomorphism on the open $V\subset U$ corresponding to curves (i.e. reduced $(2,3)$-complete 
intersections). We then obtain an essentially one-to-one correspondence between orbits $\SL(4)\cdot c\subset \Chow_{4,1}$ and orbits $\SL(5)\cdot x\subset \Delta^\nu$. Specifically, if  $c\in \varphi(V)\subset \Chow_{4,1}$, we can associate to it a unique $(2,3)$-curve $C$, and then to $C$ a cubic threefold $(X,p)$ with a marked singularity, and finally  a point $x\in \Delta^\nu$ (via the   natural map $\widetilde\Delta\to \Delta^\nu$, see \S\ref{nrquotient}). Conversely, if $(X,p)$ is not too singular (e.g. $X$ is semistable) we can reverse the process and associate a $(2,3)$-scheme $C$ and a point  $c\in \Chow{4,1}$.   One checks from the definitions that the $\SL(4)$ orbit of $c$ is identified with the $\SL(5)$ orbit of $x$.  

The only ambiguity arising  in this association between orbits of points $c\in \Chow_{4,1}$ and orbits of points in $x\in \Delta^\nu$ is when $c\in \varphi(U)\setminus \varphi(V)$. In this situation, we 
choose $C$ to be an arbitrary $(2,3)$-scheme corresponding to the cycle $c$ and then associate to it $x\in \Delta^\nu$ as before.   In fact, the only non-reduced $(2,3)$-schemes that we will need to examine are the doubled twisted cubics, and in this case the association is independent of the choices involved.  Indeed, note that if $c$ is the cycle corresponding to a double twisted cubic, the associated point $x$ corresponds to the associated chordal cubic and thus it is unambiguously defined (compare Lemma \ref{lemchordal}). Conversely, if $x\in \Delta^\nu$ is a point corresponding to a chordal cubic, we can choose an arbitrary lift $(X,p)$ and then associate $c\in \Chow_{4,1}$, which will be the cycle corresponding to the associated double twisted cubic. 

We now recall the following identifications of GIT quotients:
$$\overline{M}_4^{GIT}\cong \Delta^\nu\gquot \SL(5)\cong \Chow_{4,1}\gquot \SL(4)$$
which should be understood in terms of rings of invariant sections  (see Propositions \ref{FiniteGeneration} and \ref{GITIdentification} and Theorem \ref{chowcubic}). At the level of $\Delta^\nu$ the GIT stability is  described by Allcock's result (Theorem \ref{thmallcock}). Via the association of orbits $x\cdot \SL(5)\subset \Delta^\nu\longrightarrow c\cdot \SL(4)\subset \Chow_{4,1}$ described above (and Proposition \ref{comparesing}), we obtain stability conditions for $\overline{M}_4^{GIT}$ in terms of curves as stated in the theorem. The only remaining issue is to see that the stability conditions defined in this way  agree with  the stability conditions on $\Chow_{4,1}\gquot \SL(4)$ in the usual sense of GIT. In other words, we want to check that $c\in \Chow_{4,1}$ is semistable iff the associated point $x\in \Delta^\nu$ is semistable. 

Assume $c\in \Chow_{4,1}$ is semistable (in the standard GIT sense) and let $x\in \Delta^\nu$ the corresponding point. The semistability of $c$ is equivalent to  the existence of a non-vanishing section $\sigma\in H^0(\Chow_{4,1},A^{\otimes n})^{\SL(4)}$. Since $H^0(\Chow_{4,1},A^{\otimes n})^{\SL(4)}\cong H^0(\Delta^\nu,(L^\nu)^{\otimes n})^{\SL(5)}$ (at least after passing to suitable multiples), we obtain an $\SL(5)$-invariant section $\tau$. It is clear that $\sigma(c)\neq 0$ is equivalent to $\tau(x)\neq 0$; thus $x$ is semi-stable (as a cubic threefold)  giving that $c$ is as listed in the theorem. The converse (i.e. a semistable $x\in \Delta^\nu$ gives a semistable $c$) is also clear; this completes the proof of the semistability claims in the theorem. The only point to emphasize here is that the ambiguity (in the non-reduced case) in defining the correspondence $c\longrightarrow x$ does not cause a  problem here. Namely, as noted above, $c$ semistable gives $c\in \varphi(V)$. Then $\varphi^*A=\calO_{\bP E}(3,2)$ (see proof of Theorem \ref{chowcubic}); thus the section $\sigma$ can be regarded as an $\SL(4)$-invariant section of $\calO_{\bP E}(3,2)$. Clearly $\sigma(c)\neq 0$ is equivalent to $\sigma(C)\neq 0$ for every lift $C\in V\subset \bP E$ of $c\in \Chow_{4,1}$.

Finally, when restricted to semistable loci in $\Chow_{4,1}$, it is easy to see that stabilizer group for $c\in \Chow_{4,1}$ is the same (at least up to finite index) as the stabilizer of the associated cubic. Similarly, when restricted to the semistable loci, $c_0\in \overline {\SL(4)\cdot c}$ is equivalent to $x_0\in \overline {\SL(5)\cdot x}$ at the level of cubics (N.B. for this it suffices to check the statement for diagonal $1$-PS of $\SL(4)$ and $\SL(5)$ respectively). This allows us to conclude that the minimal orbits and stable points are as stated in the theorem.
\end{proof}

\begin{remark}\label{remcubic}
The following clarifies which schemes are parameterized by $C_{A,B}$, $C_D$, and $C_{2A_5}$. In each case we will use $Q\subseteq \mathbb P^3$ to denote the  defining quadric.    If $4A/B^2\ne 0,1$, then $C_{A,B}$ has exactly two singularities, one of type $A_3$ at $q$ (the vertex of $Q$) and one of type $A_5$ at $q'$ (a smooth point of $Q$).  If $4A/B^2=1$, then $C_{A,B}$ has exactly three singularities, one $A_1$ singularity, one $A_3$ singularity at $q$ (the vertex of $Q$), and one $A_5$ singularity at $q'$ (a smooth point of $Q$).  If $4A/B^2=1$, then $C_{A,B}$ is non-reduced, and has support equal to a rational normal curve.  $C_D$ has exactly five singularities, three of type $A_1$ and two of type $D_4$.  Finally, $C_{2A_5}$ has exactly two singularities, located at $q,q'$, both of type $A_5$ (and $Q$ is smooth); the curve has three irreducible components, each of which is a smooth rational curve, two of which are degree one (disjoint lines), and one of which has degree four and meets the other two lines (each in a single point).
\end{remark}

\begin{remark}  Allcock's theorem also describes the degenerations of the strictly semi-stable points $c\in\Chow_{4,1}$.  
   Let $C$ be a $(2,3)$-scheme with strictly semi-stable cycle $c\in \Chow_{4,1}$.  If $C$ contains a $D_4$ singularity, or lies on a rank $2$  quadric, then $c$ degenerates to the cycle associated to $C_D$.  If $C$ lies on a quadric $Q$ of rank at least $3$, and either $C$ contains an $A_5$ singularity at a smooth point of $Q$, or an $A_3$ singularity at the vertex of $Q$ (if $\rank Q=3$), then $c$ degenerates to either the cycle associated to  $C_{2A_5}$ or to the cycle associated to some $C_{A,B}$ with $4A/B^2\ne 1$.  Otherwise, $c$ degenerates to $C_{A,B}$ with $4A/B^2=1$,  a non-reduced complete intersection supported on a rational normal curve.
\end{remark}

\begin{remark}
We also have  an identification $\overline M_4^{GIT}=\bP E\gquot_{\calO(3,2)} \SL(4)$ (cf. Proposition \ref{GITIdentification}). To describe the points of $\overline M_4^{GIT}$ in terms of semi-stable points in $\bP E$, a little care is needed. The issue is that the line bundle $\calO(3,2)$ is not ample, e.g. it contracts the ribbon locus (to the double twisted cubic locus in $\Chow_{4,1}$). One natural definition for the semistable points on $\bP E$ is $(\varphi_{\mid U})^{-1}(\Chow_{4,1}^{ss})$. Alternatively, one can use the standard definition of GIT \cite{git}: a point is semistable if there is a non-vanishing invariant section $\sigma$ and the associated open set $\bP E _\sigma$ is affine. In our situation, it is easy to see that the two definitions agree for curves $C$ for which the associated cubic $X$ is semistable and such that  the orbit closure does not contain the orbit of the chordal cubic threefold (i.e. items (i.$\alpha$), (ii.$\alpha$), and (iii) from Theorem \ref{thmcubic}(2)). We will call such points essentially semistable, and denote by $(\bP E)^{ess}$ the corresponding set. For stable points, the two possible definitions agree; we let $(\bP E)^s$ be the set of stable points. The following clarifies $\overline M_4^{GIT}$ from the perspective of $\bP E$ and $\Chow_{4,1}$:
\begin{itemize}
\item[(0)] $\overline M_4^{GIT}\cong \bP E\gquot \SL(4)\cong \Chow_{4,1}\gquot \SL(4)$ is a normal projective variety.
\item[(1)] $(\bP E)^s/\SL(4)\cong \Chow_{4,1}^s/\SL(4)$ is a geometric quotient. In fact, $(\bP E)^s\cong \Chow_{4,1}^s$, and the stability is  described by Theorem \ref{thmcubic}(1).
\item[(2)] $(\bP E)^{ess}/\SL(4)$ is an orbit space in the usual sense of GIT, and the natural embedding $(\bP E)^{ess}/\SL(4)\subset \overline M_4^{GIT}$ is a one point compactification (the point corresponding to the double twisted cubic). 
\item[(3)] The boundary of $(\bP E)^s/\SL(4)$ in $\overline M_4^{GIT}$ consists of three components described by Theorem \ref{thmcubic}. 
\end{itemize}
\end{remark}

\section{Hassett--Keel Program}\label{sectHK}
The goal of the Hassett-Keel program is to provide modular interpretations of the log canonical models
$$ \overline{M}_g ( \alpha ) := \operatorname{Proj} \left( \bigoplus_{n=0}^{\infty} H^0 \left(n(K_{\overline{M}_g} + \alpha \delta ) \right) \right), \ \ \ \alpha\in [0,1]\cap \mathbb Q. $$
Hassett and Hyeon  have explicitly constructed the log minimal models $\overline{M}_g ( \alpha )$ for $\alpha \geq \frac{7}{10} - \epsilon$ (see \cite{hh,hh2}).  Hyeon and Lee have also described the next stage of the program in the specific case of genus $4$ (cf. \cite{HL2}). Finally, Fedorchuk \cite{maksymg4} has constructed the final nontrivial step in the Hassett--Keel program for $g=4$ by using GIT for $(3,3)$ curves on $\bP^1\times \bP^1$. In this section we identify the GIT quotient $\overline{M}_4^{GIT}$ with another log canonical model $\overline{M}_4 ( \alpha )$. The value of $\alpha$ corresponding to our space  $\overline{M}_4^{GIT}$ is intermediary between the slopes occurring in  \cite{HL2} and \cite{maksymg4} respectively. 

\begin{theorem}\label{thmhk}
$\overline{M}_4^{GIT} \cong \overline{M}_4 \left( \frac{5}{9} \right).$
\end{theorem}

\begin{proof}
We first note that there is a birational contraction
$$ \varphi: \overline{M}_4 \dashrightarrow \overline{M}_4^{GIT} .$$
To see that this is indeed a contraction,
recall from \S \ref{sectcubicg4} that $\varphi^{-1}$ is regular outside of a codimension two locus.
Indeed, the set $\Sigma$ of singular curves in $\overline{M}_4^{GIT}$ is an irreducible divisor.  Since the general point of $\Sigma$ corresponds to a Deligne-Mumford stable curve, we see that indeed the map $\varphi^{-1}$ is regular outside of codimension 2.  We also note here that being a birational map from a $\mathbb Q$-factorial space, we can conclude that $\phi$ extends over the generic points of each boundary divisor.

Now, let $\mathcal{L}$ denote the ample line bundle on $\overline{M}_4^{GIT}$ corresponding to the linearization $\mathcal{O}(3,2)$.  We wish to determine the numerical class of $\varphi^* \mathcal{L}$.  Write
$$ \varphi^* \mathcal{L} = a \lambda - b_0 \delta_0 - b_1 \delta_1 - b_2 \delta_2 . $$
By Proposition \ref{computeclasses}, we know that $a = 9$, $b_0 = 1$.  To compute the coefficient $b_1$, consider the curve $Z \subset \overline{M}_4$ obtained by gluing a fixed element of $M_{3,1}$ to a standard pencil of elliptic curves.  Since cuspidal curves are stable in $\overline{M}_4^{GIT}$,  the arguments in \S \ref{sectcubicg4} show that the map $\varphi$ is regular and constant on $Z$, hence $\varphi^* \mathcal{L} \cdot Z = 0$.  It is well known that $Z.\lambda=1$, $Z.\delta_0=12$, $Z.\delta_1=-1$, and $Z.\delta_2=0$.  It follows that
$$ a - 12b_0 + b_1 = 0 ,$$
so $b_1 = 3$.

To determine the coefficient $b_2$, consider the gluing map
$$ g: \overline{M}_{2,1} \to \overline{M}_4 $$
given by gluing a fixed general, genus two, marked curve  at the respective marked points.  We note that if $(C,p) \in \overline{M}_{2,1}$ is integral and $p$ is not a Weierstrass point of $C$, then associated to  $g(C)$ is a GIT semi-stable $(2,3)$-curve in $\mathbb P^3$ that is in the orbit of the curve $C_{2A_5}$ (see Example \ref{exaC2A5}; recall the curve $C_{2A_5}$  consists of three components: two lines and one component of degree 4, meeting the other two in $A_5$ singularities).  Doing this for one parameter families, as in Example \ref{exaC2A5}, it follows that this describes the  extension of $\varphi$ over the generic point of  $\delta_2$.

 In particular, the map $\varphi \circ g$ is regular and constant along the complement of $\delta_1 \cup W^1_2$, where $W^1_2 \subset \overline{M}_{2,1}$ is the Weierstrass divisor.  It follows that $g^* \varphi^* \mathcal{L}$ is supported along the union of these two divisors, and hence on $\overline M_{2,1}$,
$$ g^* \varphi^* \mathcal{L} = b_2 \omega + 9 \lambda - \delta_0 - 3 \delta_1 = b_2 \omega - \lambda - 2 \delta_1 \sim 3 \omega - \lambda \pmod{\delta_1}.$$
We conclude that $b_2 = 3$.  Note that in the computation above we are using the so-called genus $2$ $\lambda$-formula, $10\lambda=\delta_0+2\delta_1$, and properties of pull-backs of divisor classes (see e.g. Morrison \cite[Formula 1.52, p.35 and Lemma 1.26, p. 18]{morrison}).

Finally, note that, since $\delta_1$ and $\delta_2$ are $\varphi$-exceptional, we have
\begin{eqnarray*}
 H^0 ( \overline{M}_4 , \varphi^* \mathcal{L}^{\otimes n})& =& H^0 ( \overline{M}_4 , n(9 \lambda - \delta_0 - 3 \delta_1 - 3 \delta_2 ))\\
 &\cong& H^0 ( \overline{M}_4 , n((9 \lambda - \delta_0 - 3 \delta_1 - 3 \delta_2) + 2( \delta_1 + \delta_2 )))\\
 & =& H^0 ( \overline{M}_4 , n(9 \lambda - \delta_0 - \delta_1 - \delta_2)) .
 \end{eqnarray*}
Thus, $\varphi$ being a birational contraction,
\begin{eqnarray*}
 \overline{M}_4^{GIT} &=& \Proj \bigoplus_n H^0 ( \overline{M}_4 , \varphi^* \mathcal{L}^{\otimes n}) \\
& \cong& \operatorname{Proj} \bigoplus_n H^0 \left( \overline{M}_4 , \left(K_{\overline{M}_4} + \frac{5}{9} \delta \right)^{\otimes n}\right) \\
&=& \overline{M}_4 \left( 5/9 \right).
\end{eqnarray*}
\end{proof}

We note that all of the singularities appearing in Theorem \ref{thmcubic} (i.e. $A_1,\dots,A_4$ are stable; $A_5$ and $D_4$ as boundary cases)  are as predicted in \cite{AFS}.  Indeed, it is expected that curves with $A_n$ singularities should appear in $\overline{M}_4 \left( \frac{5}{9} \right)$ for all $n \leq 4$, with $A_2$ singularities replacing elliptic tails, $A_3$ singularities replacing elliptic bridges, and $A_4$ singularities replacing Weierstrass 2-tails.  A local description of a natural resolution of the rational map $\overline{M}_4^{GIT}\dashrightarrow \overline{M}_4$ along the $A_2,\dots,A_4$ loci is given in \cite{ade} (see esp. \S4.2). In addition, there is a unique   closed orbit of a strictly semi-stable curve lying on a smooth quadric, namely the orbit of $C_{2A_5}$.
As noted above, this curve is the image of the generic point of $\delta_2$. Similarly, curves with $D_4$ singularities are also predicted to appear in the Hassett--Keel program at precisely the critical value $\frac{5}{9}$ (cf. \cite{AFS}), as a replacement for curves with an elliptic component meeting the rest of the curve in 3 points.  Such curves are referred to as elliptic triboroughs in \cite{AFS}. The closed orbit corresponding to this case is that of the curve $C_D$ from   Theorem \ref{thmcubic}.

Finally, the only non-reduced and strictly semistable curve is $C_{A,B}$ for $4A/B^2=1$ (see Rem. \ref{remcubic}); it comes from projection from the chordal cubic. This curve  is an example of a ribbon, i.e. a double structure on a rational normal curve.  The standard reference on ribbons is \cite{BE}, where they are introduced as the canonical limit of a family of curves degenerating to a hyperelliptic curve.  They are expected to appear in a flip of the hyperelliptic locus of $\overline{M}_g$.   The existence and construction of this flip is currently an open problem (with genus $4$ as the first instance) in the Hassett--Keel program. We expect that a geometric consequence of the comparison of $\overline{M}_4^{GIT}$ to the ball quotient model (Thm. \ref{mainthm}), will be a construction of the hyperelliptic flip in the genus $4$ case. This will be discussed elsewhere.

\section{Ball quotient model for  the moduli of genus $4$ curves}\label{sectball}
A ball quotient model for $\overline{M}_4$ was constructed by Kondo in \cite{k2}. We briefly review the construction below.  We then establish some  facts about the discriminant hyperplane arrangement and the Baily--Borel compactification. We conclude with a result about the polarization of the ball quotient (Thm. \ref{thmballpol}).

\subsection{Kondo's construction}\label{sectconst}
A smooth non-hyperelliptic genus $4$ curve is contained in a unique quadric surface $Q$. The cyclic triple cover of $Q$ branched along $C$ is a $K3$ surface $S$. Conversely, $S$ together with the covering automorphism recovers $C$. It is well known that, via the period map, the moduli space of $K3$ surfaces is a locally symmetric variety; it is the quotient of a Type $IV$ bounded symmetric domain by the monodromy group. Taking into account the covering automorphism (see  \cite{dk} for a discussion of the general theory of ``eigenperiods''), one obtains that the moduli space of genus $4$ curves is birational to a $9$-dimensional ball quotient $\calB_9/\Gamma$.  More precisely, Kondo \cite{k2} proved the following:

\begin{theorem}[{Kondo}]\label{thmkondo}
The construction described above induces an isomorphism
\begin{equation}\label{eqkondo}
\Phi_0:M_4^{ns}\tilde\to \left(\calB_9\setminus(\calH_{v}\cup \calH_{n}\cup \calH_{h})\right)/\Gamma
\end{equation}
between the moduli of non-special, genus $4$ curves and the quotient of the  complement of a hyperplane arrangement in a $9$-dimensional complex ball.
Moreover, $\Phi_0$ extends along the vanishing theta locus $V$ with image in $\calH_{v}/\Gamma$, and at the generic point of $\Delta_0$ with image in $\calH_n/\Gamma$. The hyperelliptic Heegner divisor $\calH_h/\Gamma$ parameterizes pairs $(C,\sigma)$ with $C$ a hyperelliptic genus $4$ curve and $\sigma\in g^1_2$.
\end{theorem}
\begin{proof}
The isomorphism $M_4^{nh}\tilde\to \left(\calB_9\setminus(\calH_{n}\cup \calH_{h})\right)/\Gamma$ is \cite[Thm. 1]{k2}; the behavior along the vanishing theta locus $V$ is discussed in \cite[Rem. 4]{k2}. The results about the nodal and hyperelliptic locus are  \cite[Thm. 2]{k2}.
\end{proof}

A similar construction involving the period map for cubic fourfolds was used by  Allcock--Carlson--Toledo \cite{act} and Looijenga--Swierstra \cite{ls} to prove that the moduli of cubic $3$-folds is birational to a $10$-dimensional ball quotient.
\begin{theorem}[{\cite{act},\cite{ls}}]\label{thmact}
 Let $M_{\operatorname{cubic}}^{sm}$ be the moduli space of smooth cubic threefolds. Then
$$M_{\operatorname{cubic}}^{sm}\cong (\calB_{10}\setminus (\calH_{0}\cup \calH_{\infty}))/\Gamma'.$$
The hyperplane arrangements $\calH_0$ and $\calH_{\infty}$ correspond to the singular cubics, and  degenerations to the chordal cubic.
\end{theorem}

As explained in section \ref{sectprelim}, the moduli space of genus $4$ curves is closely related to the discriminant divisor in the moduli of cubics. In the context of ball quotient models, the relationship can be made very precise.

\begin{proposition}\label{compareballqt}
Kondo's ball quotient model for $\overline{M}_4$ is compatible with the ball quotient model for the moduli of cubics,  in the sense that there exists a natural map
$$\calB_9/\Gamma\to \calH_0/\Gamma'\subset \calB_{10}/\Gamma',$$
which is a normalization morphism onto the image.
\end{proposition}
\begin{proof}
Generally speaking, a group embedding $G=\SU(1,n)\subset G'=\SU(1,n+1)$ defines a totally geodesic embedding $\calB_n\subset \calB_{n+1}$. Assuming $\Gamma'\subset G'$ is an arithmetic subgroup and $G\subset G'$ is defined over $\bQ$, we can define $\Gamma=\Gamma'\cap G$ and $\calH$ to be the hyperplane arrangement obtained by considering all $\Gamma'$-translates of $\calB_n$ in $\calB_{n+1}$. Clearly, there is a morphism $\calB_n/\Gamma\to \calB_{n+1}/\Gamma'$ which is birational onto the image $\calH/\Gamma'$. The arithmeticity assumption assures that the morphism is finite; thus a normalization onto the image. Given these general facts, the result follows from the discussion of \cite[Ch. 5]{act} (esp. \cite[Thm. 5.1]{act}).
\end{proof}

We note that the nodal divisor $\calH_0/\Gamma'$ in the moduli of cubics $\calB_{10}/\Gamma'$ decomposes into two irreducible divisors $\calH_n/\Gamma$ and $\calH_v/\Gamma$ when restricted to the moduli of genus four  curves $\calB_9/\Gamma$. The geometric meaning is clear by Prop. \ref{comparesing}: the self-intersection of the discriminant divisor for cubics corresponds to two $A_1$ singularities or to an $A_2$ singularity, giving  the two cases. On the other hand, the chordal divisor $\calH_\infty/\Gamma'$
restricts to the hyperelliptic divisor $\calH_h/\Gamma$. Geometrically, $\calH_\infty/\Gamma'$ corresponds to hyperelliptic genus $5$ curves (see \cite[Ch. 4]{act}), and $\calH_h/\Gamma$ to their degenerations.

\subsection{The arithmetic of the hyperplane arrangement} In this section we discuss some basic facts about the hyperplane arrangements $\calH_n,\calH_v$, and $\calH_h$. These results are arithmetic in nature, and are standard applications of lattice theory (see esp. Nikulin \cite{nikulin} and Allcock \cite{allcockaut}). For some relevant background to our situation, we refer the reader to \cite{act}, \cite{k2}, \cite{dk}, and \cite{scattone}.

\subsubsection{Preliminaries} To start, we recall that an Eisenstein lattice $L^\calE$ is a free module over the ring of Eisenstein integers $\calE=\bZ\left[\frac{-1+i\sqrt{3}}{2}\right]$ together with a Hermitian form. The arithmetic groups $\Gamma, \Gamma'$ of Theorems \ref{thmkondo} and \ref{thmact}, as well as the corresponding hyperplane arrangements are described in terms of certain Eisenstein lattices of hyperbolic signatures (e.g. \cite[\S7]{act}). Here we prefer a slightly indirect, but more familiar description, using standard lattices. Namely, an Eisenstein lattice $L^\calE$ is equivalent to a standard lattice $L$ (i.e. $\bZ$-module with symmetric bilinear form) together with a fixed-point-free isometry $\rho$ of order $3$. Simply put, $L$ is the  $\bZ$-module underlying $L^\calE$ together with the real part (suitably scaled) of the hermitian form. The isometry $\rho$ corresponds to the multiplication by a root of unity. The rank of $L$ and the signature are double
 the rank and signature of $L^{\calE}$. Thus, if $L^\calE$ is hyperbolic of signature $(1,n)$, then the associated $\bZ$-lattice $L$ is of signature $(2,2n)$. At the level of symmetric domains,  the $n$-dimensional complex ball $\calB_n$ (associated to $L^\calE$) has a totally geodesic embedding into the Type $IV$ domain $\calD_{2n}$ (associated to $L$) of dimension $2n$. Furthermore, it is standard to recover the subdomain $\calB_n\subset \calD_{2n}$ using the isometry $\rho$ (see \cite{dk}). The monodromy group $\Gamma$ is described in terms of the subgroup $\widetilde \Gamma$ of elements commuting with $\rho$ in the corresponding orthogonal group (see \cite[p. 387]{k2}).

In our situation, the (integral) lattices associated to the cases of genus $4$ curves and cubic threefolds are $T=E_8^{\oplus2}\oplus U\oplus U(3)$ and $T'=E_8^{\oplus2}\oplus A_2\oplus U\oplus U$ respectively. In fact, $T$ and $T'$ are the transcendental lattices of the $K3$ surface or cubic fourfolds occurring in the constructions of \cite{k2} and \cite{act} respectively. The covering automorphism involved in the construction (see the first paragraph of \S\ref{sectconst}) induces a fixed point free isometry $\rho$ (and $\rho'$) for these lattices. As explained in the previous paragraph, this is equivalent to an Eisenstein lattice structure $T^\calE$ on $T$ (and similarly for $T'$). A useful trick for understanding the arithmetic aspects of our examples is the following standard result :
\begin{lemma}\label{lemma1}
The lattice $T=E_8^{\oplus2}\oplus U\oplus U(3)$ (resp. $T'$) has a primitive embedding into the unimodular lattice $\Lambda=E_8^{\oplus3}\oplus U\oplus U$ with orthogonal complement $R=E_6\oplus A_2$ (resp. $R'=E_6$). Furthermore, the isometry $\rho$ of $T$ extends to a fixed-point-free isometry of $\Lambda$ (and similarly for $T'$). \qed
\end{lemma}

We also note the following result.
\begin{lemma}\label{lemma2}
Assume that $M$ is a negative definite root lattice such that there exists a fixed-point-free isometry $\rho\in O(M)$ of order $3$. Then, $M$ is a direct sum of $A_2,D_4,E_6,E_8$ summands and $\rho$ preserves these summands.
\end{lemma}
\begin{proof}
A negative definite root lattice is a direct sum of $ADE$ summands. We claim that $\rho$ preserves these summands. This statement would follow provided that $\delta.\rho(\delta)\neq 0$ for all roots $\delta$. Note that $\delta+\rho(\delta)+\rho^2(\delta)$ is $\rho$-invariant, and thus it has to be $0$. Then, using $(\rho^2(\delta))^2=-2$, we have $\delta.\rho(\delta)=1$, and the claim follows. Consequently, we are reduced to the case $M$ is an irreducible root lattice. This case is standard, e.g. \cite{carter}.
\end{proof}

\subsubsection{Description of the hyperplane arrangements} In the case of $K3$ surfaces (or cubic fourfolds), the discriminant hyperplanes correspond to the situations where the  lattice of algebraic cycles acquires additional cycles (typically $(-2)$-classes).  The situation is similar in the Eisenstein lattice case (see \cite[\S3]{k2}), except that one has to take into account the isometry $\rho$. Namely, a hyperplane in $\calB_9$ corresponds to a codimension two locus in the associated Type $IV$ domain $\calD_{18}=\calD_T$. These codimension two loci are determined by
sublattices $M\subset T$ of signature $(2,2n-2)$ (here $n=9$). Specifically, we have
$$\calD_M=\{\omega\in \bP(M_\bC)\mid \omega.\omega=0,\omega.\omega>0\}\subset \calD_T=\{\omega\in \bP(T_\bC)\mid\dots\}.$$
If $M$ is invariant with respect to $\rho$, the above inclusion determines an inclusion of complex balls $\calB_{M^\calE}\subset \calB_{T^\calE}$ such that $\calB_{M^\calE}$ is a hyperplane in a suitable embedding of the $n$-dimensional ball $\calB_{T^\calE}$ in $\bP^n$.

Lemma \ref{lemma1} gives an embedding of the $18$-dimensional Type $IV$ domain $\calD_T$ into a $26$-dimensional Type $IV$ domain $\calD_\Lambda$. Thus, by the above discussion,  we can view a hyperplane in $\calB_9$ as given by an embedding of lattices (invariant w.r.t. $\rho$):
\begin{equation}\label{auxeq}
M\subset T\subset \Lambda.
\end{equation}
The embedding $T\subset \Lambda$ is fixed by Lemma \ref{lemma1}. Thus, we can view \eqref{auxeq} as equivalent to the  embeddings
\begin{equation}\label{auxeq2}
T^\perp_\Lambda=R\cong E_6\oplus A_2\subset M^\perp_\Lambda\subset \Lambda\cong E_8^{\oplus 2}\oplus U^{\oplus2}.
\end{equation}
Note that $M_\Lambda^\perp$ is negative definite of rank $10$. Additionally, from Kondo \cite{k2}, one sees that  the lattice $M_\Lambda^\perp$, corresponding to the hyperplane arrangements of Theorem \ref{thmkondo}, contains additional roots. Combining this with Lemma \ref{lemma2}, we conclude:
\begin{proposition}\label{hypdescr}
With notation as above (esp. \eqref{auxeq2}), the hyperplanes in $\calH_v$, $\calH_n$, $\calH_h$ (from Thm. \ref{thmkondo}) correspond to the cases $M_\Lambda^\perp$ being isometric to $D_4\oplus E_6$, $A_2\oplus A_2\oplus E_6$, and $A_2\oplus E_8$ respectively. \qed
\end{proposition}

In particular, using this description of the hyperplane arrangements, we conclude:
\begin{corollary}\label{noint}
The hyperplanes $\calH_h$ do not intersect in the interior of $\calB_9$.
\end{corollary}
\begin{proof}
An intersection of two hyperplanes $H_i$, for $i=1,2$  from $\calH_h$ would correspond to  lattice embeddings $M''\subset M_i\subset T$, with $M_i\subset T$ giving the hyperplanes $H_i$, and with $M''$ of signature $(2,14)$.  Dually, we have $T^\perp_\Lambda\cong E_6\oplus A_2\subset (M_i)^\perp_\Lambda\cong E_8\oplus A_2\subset (M'')_\Lambda^\perp$ with $(M'')_\Lambda^\perp$ negative definite of rank $12$. Thus, $(M'')_\Lambda^\perp$ contains two different $E_8$ extensions of $E_6$, which is a clear contradiction of the fact that $(M'')_\Lambda^\perp$ is negative definite.
\end{proof}

\begin{remark}
Completely analogous results hold also for cubic threefolds. There, $\calH_0$ and $\calH_\infty$ correspond to the case $A_2\oplus E_6$ and $E_8$ respectively. In particular, the statements of the last paragraph of \S\ref{sectconst} have a clear arithmetic explanation.
\end{remark}
\subsection{The Baily--Borel compactification} Based on the discussion of the previous subsection, we can compute the Baily--Borel compactification of the ball quotient model. Similar computations in the case of $K3$ surfaces are done in \cite{scattone}.
\begin{theorem}\label{propbb}
The Baily-Borel compactification $(\calB_9/\Gamma)^*$
of Kondo's ball quotient model $\calB_9/\Gamma$ has three cusps labeled $c_{E_6^{\oplus 2}\oplus A_2^{\oplus 2}}$, $c_{E_6\oplus A_2\oplus E_8}$, and $c_{E_8^{\oplus 2}}$. The hyperelliptic divisor passes only through the cusp $c_{E_6\oplus A_2\oplus E_8}$.
\end{theorem}
\begin{proof}
The classification of cusps for the Baily--Borel compactification is equivalent to the classification of isotropic vectors in the lattice $T^\calE$. This is equivalent to the classification of rank $2$ isotropic sublattices $E$ of $T$ that are invariant with respect to $\rho$. A basic invariant of $E\subset T$ is the negative definite lattice $E^\perp_T/E$ (which comes endowed with fixed-point-free isometry).  In many case, the classification of   $E^\perp_T/E$ is equivalent to the classification of $E$.

The standard technique for studying $E^\perp_T/E$ (see \cite[\S5]{scattone}) is to use the embedding $T\subset \Lambda$ given by Lemma \ref{lemma1}. Namely, one shows that $E^\perp_\Lambda/E$ is unimodular, and thus it is one of the $24$ Niemeier  lattices (see \cite[\S18.4]{conway}). Then $E^\perp_T/E$ is a sublattice of $E^\perp_\Lambda/E$ with orthogonal complement $R=E_6\oplus A_2$. It is easy to classify the possible embeddings of $R$ into the unimodular lattice  $E^\perp_\Lambda/E$. Namely, in our situation, keeping track of the isometry $\rho$ and using Lemma \ref{lemma2}, we get $3$ cases for the embedding $R\subset E^\perp_\Lambda/E$:
\begin{itemize}
\item either (i) $E_6\oplus A_2\subset E_6^{\oplus 4}$ (\cite[p. 438, Case XIV]{conway});
\item or $E_6\oplus A_2\subset E_8^{\oplus 3}$ (\cite[p. 438, Case XV]{conway}) with two subcases:
\begin{itemize}
\item[(ii)] $E_6$ and $A_2$ embed in different copies of $E_8$,
\item[(iii)] $E_6$ and $A_2$ embed in the same copy of $E_8$.
\end{itemize}
\end{itemize}
We conclude that the root sublattice contained in $E^\perp_T/E$ is $E_6^{\oplus 2}\oplus A_2^{\oplus 2}$, $E_6\oplus A_2\oplus E_8$, and $E_8^{\oplus 2}$  for the cases (i), (ii), and (iii) respectively. It is easy to see that the associated root sublattices  actually determine the isometry class of $E^\perp_T/E$ (even as Eisenstein lattice). Finally, one checks that the invariant $E^\perp_T/E$ classifies the isotropic subspace $E\subset T$ (see \cite{scattone} for related examples). We conclude that there are three cusps in the Baily--Borel compactification and we label them by the root sublattice listed above.

Finally, if a hyperplane from $\calH_h$ passes through the cusp given by $E$, then we will have $E\subset M\subset T\subset \Lambda$ (with notation as in \eqref{auxeq}). Applying Prop. \ref{hypdescr}, this is clearly only possible in the case (ii) above.
\end{proof}

\begin{remark}
The same argument can be used to show that in the case of cubic threefolds, the Baily--Borel compactification $(\calB_{10}/\Gamma')^*$ has two cusps, which using the notation as above would be labeled $c'_{E_6^{\oplus 3}}$ and $c'_{E_8^{\oplus 2}\oplus A_2}$ (compare \cite[Thm. 3.10]{act} and \cite[Rem. 3.2]{ls}). By restricting to genus $4$ curves and passing to normalization (see Proposition \ref{compareballqt}), the cusp $c'_{E_8^{\oplus 2}\oplus A_2}$ separates into $c_{E_6\oplus A_2\oplus E_8}$, and $c_{E_8^{\oplus 2}}$.
\end{remark}

\subsection{The polarization of the ball quotient model $\calB_9/\Gamma$}\label{sectautomorphic}
As discussed above, in $(\calB/\Gamma)^*$ there are $3$ special  Heegner divisors $H_v=\calH_v/\Gamma$, $H_n=\calH_b/\Gamma$, and $H_h=\calH_h/\Gamma$. These are irreducible Weil divisors, but not $\bQ$-Cartier due the cusps of $(\calB/\Gamma)^*$. Here we establish that a certain linear combination of $H_v,H_n$, and $H_h$ gives a natural polarization on $(\calB/\Gamma)^*$. The technique of proving this is by now standard and it is based on Borcherds' construction of an automorphic form on the $26$ dimensional Type $IV$ domain $\calD_\Lambda$ (cf. \cite{borcherds0,bkpsb}). The closely related case of cubic threefolds is discussed in full detail in \cite[\S7.2]{cml}

 \begin{theorem}\label{thmballpol}
 The divisor $H_n+\frac{9}{2}H_v+14H_h$ is an ample $\bQ$-Cartier divisor on the ball quotient $(\calB/\Gamma)^*$.
 \end{theorem}
 \begin{proof}
 We consider Borcherds' automorphic form $\widetilde \psi$ of weight $12$ on $\calD_\Lambda$ that vanishes with order $1$ on all hyperplanes orthogonal to $(-2)$-classes in $\Lambda$. As explained in \cite{bkpsb},  the automorphic form $\widetilde{\psi}$ induces an automorphic form $\overline{\psi}$ on  $\calD_T$. Then $\overline{\psi}$ induces an automorphic form $\psi$ on $\calB_{T^\calE}$, whose associated divisor is the divisor mentioned in the theorem.  The vanishing order of $\overline{\psi}$ along the loci $H_n,H_v,H_h$ at the level of $\calD_T$ is equal to half the number of roots in $M_\Lambda^\perp$ minus the number of roots in $R$ (see the notation from \eqref{auxeq2}, Prop. \ref{hypdescr}). Thus, we obtain orders $\frac{1}{2}\#\textrm{roots}(A_2)=2$,  $\frac{1}{2}\left(\#\textrm{roots}(D_4)-\#\textrm{roots}(A_2)\right)=9$, and $\frac{1}{2}\left(\#\textrm{roots}(E_8)-\#\textrm{roots}(E_6)\right)=84$ for $\overline{\psi}$ on $H_n,H_v,H_h$ respectively.
Finally, to compute $\divisor(\psi)$ at the level of $(\calB/\Gamma)^*$ one needs to take into account the action of the monodromy group $\Gamma$, and divide by the corresponding ramification orders: $3$, $2$, $6$ respectively (compare \cite[Lem 1.4, Lem. 3.3]{act}).
 \end{proof}

\section{The comparison of the GIT and ball quotient models}\label{sectcompare}
In this section, we show that the GIT quotient $\overline{M}_4^{GIT}$ considered in section \ref{sectgit} is closely related to Kondo's ball quotient model $B_9/\Gamma$. The precise relationship is given by Thm. \ref{mainthm}, which follows from the theory of Looijenga \cite{looijengacompact1} (reviewed in \S\ref{sectl} below), the properties of the Kondo's period map (induced by those of the period map for $K3$s), and the computation of polarizations for $\overline{M}_4^{GIT}$ and $B_9/\Gamma$ (see Proposition \ref{GITIdentification} and Theorem  \ref{thmballpol} respectively). We close with a discussion of some of consequences of Thm. \ref{mainthm}.

\subsection{A review of Looijenga's theory}\label{sectl}
Looijenga \cite{looijengacompact1,looijengacompact} developed a framework to compare GIT quotients with arithmetic quotients of Type $I_{1,n}$ (complex balls) and Type $IV$ domains. This framework  was successfully applied in several situations arising from geometry (e.g. cubic threefolds \cite{act,ls}, and genus $3$ curves \cite{l2}). Here we briefly recall the basic ingredients:
\begin{itemize}
\item[(1)] Assume we are given $M$, an open subset of a projective variety $\overline{M}$, and  $L$  an ample $\bQ$-Cartier divisor on $\overline{M}$. Typically, $\overline{M}$ will be a GIT quotient, and $M$ the open subset corresponding to the stable points.
\item[(2)] Let $\calD$ be a complex ball or a Type $IV$ domain, $\Gamma$ an arithmetic group acting on $\calD$, and  $\calH$ a $\Gamma$-invariant arithmetic arrangement of hyperplanes in $\calD$ (see \cite[Ex. 1.8, 1.9]{looijengacompact1}). Assume we are given a morphism
$$\Phi: M\to \calD/\Gamma$$
such that
\begin{itemize}
\item[i)] $\Phi:M\to \calD/\Gamma$ is injective;
\item[ii)] $\im(\Phi)=(\calD\setminus \calH)/\Gamma$ (the complement of an arithmetic hyperplane arrangement).
\end{itemize}
In practice, $\Phi$ is a period map and the two conditions  above follow from a Torelli type theorem and a properness result for the period map (see also \cite{ls2} for a related discussion). We note that while the period map is a priori defined in the analytic category, in our situation, the algebraicity follows from Baily--Borel theorem (i.e. $\calD/\Gamma$ is quasi-projective, with projective compactification $(\calD/\Gamma)^*$) and Borel's extension theorem (i.e. $\Phi$ is an algebraic morphism).
\item[(3)]  $(\calD/\Gamma)^*$ has a natural polarization $L'$ (given by the Baily--Borel construction) and $\overline{M}$ carries a polarization $L$ by assumption. One requires
\begin{equation}\label{linagree}
(L)_{\mid M}=(L')_{(\calD\setminus \calH)/\Gamma};
\end{equation}
i.e. the two polarizations agree on the common open subsets.
\item[(4)] Finally, assume that $\overline{M}\setminus M$ has codimension at least $2$ in $\overline{M}$ and that any intersection of  hyperplanes from $\calH$ has dimension at least $1$. This is a mild condition in practice.
\end{itemize}
If these assumptions are satisfied, Looijenga has constructed an explicit birational modification $\overline{\calD/\Gamma}_\calH$ of $(\calD/\Gamma)^*$, which leaves   $(\calD\setminus \calH)/\Gamma$ unchanged, such that
\begin{equation}\label{lthm}
\overline{M}\cong \overline{\calD/\Gamma}_\calH \end{equation}
(see \cite[Thm. 7.1]{looijengacompact1}).

The space $\overline{\calD/\Gamma}_\calH$ can be described, in analogy with the Baily--Borel compactification, as the $\Proj$ of the ring of meromorphic forms with poles along $\calH/\Gamma$. More explicitly, $\overline{\calD/\Gamma}_\calH$ is constructed in the following three steps:
\begin{itemize}
\item[(A)] First, construct a small blow-up $\widehat{\calD/\Gamma}_\calH$ of the boundary of $(\calD/\Gamma)^*$ such that the Weil divisor $\calH/\Gamma$ becomes $\bQ$-Cartier (see \cite[\S4.3, and Lem. 5.2]{looijengacompact1}).
\item[(B)] Then, consider a blow-up $\widetilde{\calD/\Gamma}_\calH$  of $\widehat{\calD/\Gamma}_\calH$  such that the hyperplane arrangement becomes normal crossings (in an orbifold sense; see \cite[\S2.2]{looijengacompact}).
\item[(C)] Finally, $\overline{\calD/\Gamma}_\calH$ is obtained by contracting the hyperplane arrangement resulting from (B) in the opposite direction (see \cite[\S3, Thm. 3.9, and Thm. 5.7]{looijengacompact1})
\end{itemize}
The net effect of the birational modification $\overline{\calD/\Gamma}_\calH$ of $(\calD/\Gamma)^*$ is to replace intersection strata from $\calH$ with strata of complementary dimension.  In particular, if the hyperplanes from $\calH$ do not intersect in $\calD$, $\overline{\calD/\Gamma}_\calH$ is essentially the contraction of the divisor $\calH/\Gamma$. More precisely, after a small modification $\widehat{\calD/\Gamma}_\calH$ of the Baily--Borel boundary (cf. step (A)),  the divisor $\calH/\Gamma$ becomes $\bQ$-Cartier and can be contracted to a point.

\begin{remark} The arrangement $\calH$ can be empty, in which case   $\overline{\calD/\Gamma}_\calH\cong (\calD/\Gamma)^*$. A geometric situation when this happens is the case of cubic surfaces: the GIT quotient for cubic surfaces is isomorphic to the Baily--Borel compactification of a ball quotient (see \cite{act0}).
\end{remark}

\subsection{The main result for genus $4$ curves} At this point, we can state our main result that compares Kondo's ball quotient model for $\overline{M}_4$ to the GIT quotient $\overline{M}_4^{GIT}$. As already mentioned, this result has close analogues in the case of the moduli space of cubic threefolds (see \cite{ls} and \cite{act}) and genus $3$ curves (see \cite{l2}, \cite{artebani}).

\begin{theorem}\label{mainthm}
The GIT quotient $\overline{M}_{4}^{GIT}$  for canonical genus $4$ curves is isomorphic to Looijenga's modification of Kondo's ball quotient model $\calB_9/\Gamma$ associated to the hyperelliptic hyperplane arrangement $\calH_h$:
\begin{equation}\label{mainisom}
\overline{M}_{4}^{GIT}\cong \overline{\calB_9/\Gamma}_{\calH_h}.
\end{equation}
More explicitly, there exists a diagram
\begin{equation}
 \xymatrix{
&&\widehat{M}_{4} \ar@{->}[ld]\ar@{->}[rd]&&\\
&\overline{M}_{4}^{GIT}\ar@{-->}[rr]&&(\calB_{9}/\Gamma)^* }
\end{equation}
such that
\begin{itemize}
\item[i)]  $\widehat{M}_{4}$ is a small blow-up of $(\calB_{9}/\Gamma)^*$, which replaces the cusp $c_{E_6\oplus A_2\oplus E_8}$ in the Baily-Borel compactification $(\calB_{9}/\Gamma)^*$ by a rational curve;
\item[ii)] $\widehat{M}_{4}\to\overline{M}_{4}^{GIT} $ contracts the strict transform of the hyperelliptic divisor $\calH_h/\Gamma$ to the point corresponding to  the double twisted cubic (see Thm. \ref{thmcubic}, Rem. \ref{remcubic}).
\end{itemize}
\end{theorem}
\begin{proof}
According to Thm. \ref{chowcubic}, a description of the GIT quotient is given by Thm. \ref{thmcubic}. We define  $M\subset \overline{M}_{4}^{GIT}$ to be the subset corresponding to  stable genus $4$ curves. From Thm. \ref{thmcubic}, $M$ parameterizes curves $C$ which are $(2,3)$-complete intersections with the following two properties:
\begin{itemize}
\item the unique quadric $Q$ containing $C$ is either smooth or a quadric cone;
\item the possible singularities of $C$ are of type $A_1,\dots, A_4$ at a smooth point of $Q$ or $A_1, A_2$ at the vertex of $Q$ (in the cone case).
\end{itemize}
The key observation now is that the surface obtained as the triple cover $X$ of $Q$ branched along $C$ has at worst $ADE$ singularities. For instance, $p\in C$ is a smooth point of $Q$, and a singular point of type $A_1,\dots,A_4$ for $C$, the the associated singularity of $X$ is of type $A_2, D_4, E_6,E_8$ respectively; explicitly, a local equation of type $A_4: x^2+y^5$, becomes a local equation of type $E_8:x^2+y^5+z^3$ after a triple cover. It follows that Kondo's construction can be extended over all of $M$. Thus, there is a period map
$$\Phi:M\to \calB_9/\Gamma.$$
An easy application of the Torelli theorem for $K3$ surfaces, gives that $\Phi$ is an embedding. Using basic facts about linear systems on $K3$s, it is immediate to see (by the same argument as in \cite[p. 393-394]{k2}) that $\im(\Phi)$ misses the hyperplane arrangement $\calH_h$, corresponding to the so called unigonal surfaces. Finally, using the
 surjectivity of the period map for $K3$ surfaces, we conclude
\begin{equation}\label{temp}
M\cong (\calB_9\setminus \calH_h)/\Gamma
\end{equation}
 (using the same arguments as in \cite[Proof of Thm. 1]{k2}; for a very similar situation see also \cite[Prop. 4.2]{l2}).

We are now in the situation described in \S\ref{sectl}. To conclude, we need in addition to the isomorphism \eqref{temp}, an identification of the polarizations. Via the birational map $\overline{M}_4^{GIT}\dashrightarrow (\calB_9/\Gamma)^*$ it is clear that the vanishing theta divisor $V$ and the discriminant divisor $\Sigma$ in $\overline{M}_4^{GIT}$ correspond to the Heegner divisors $H_v$ and $H_n$ respectively (the period map being an isomorphism at the generic points of those divisors, compare Thm. \ref{thmkondo}). From  Thm. \ref{thmballpol},  we obtain that the natural polarization $L'$ on $(B_9/\Gamma)^*$ satisfies
$$(L')_{(\calB_9/\Gamma)^*}=\Sigma+\frac{9}{2} V\sim 9\lambda-\delta \mod H_h$$
On the other hand, from Prop. \ref{computeclasses}, the linearization on $\overline{M}_4^{GIT}$ is again  proportional to $9\lambda-\delta$. It follows that on the open subset $(\calB_9\setminus \calH_h)/\Gamma$ the two linearizations agree.
Also, in our situation, the codimension conditions (cf. \S\ref{sectl}(4)) are trivially satisfied. Thus, the theorem follows from \cite[Thm. 7.1]{looijengacompact1} (see also \eqref{lthm} above).

Finally, the explicit description is a consequence of the general construction of $\overline{\calB_9/\Gamma}_{\calH_h}$ and the second part of Theorem \ref{propbb}. \end{proof}

We close by briefly discussing the geometric meaning of Theorem \ref{mainthm}.  The theorem says that ball quotient model and the GIT quotient of the Chow variety agree on the locus of stable points, i.e. $(2,3)$-complete intersections with mild singularities (up to $A_4$, see Thm. \ref{thmcubic}). At the boundary, the Baily--Borel compactification $(\calB_9/\Gamma)^*$  almost agrees with the GIT compactification $\overline{M}_4^{GIT}$. Specifically, the boundary of $(\calB_9/\Gamma)^*$  consists of three points, the cusps $c_{E_6^{\oplus 2}\oplus A_2^{\oplus 2}}$, $c_{E_6\oplus A_2\oplus E_8}$, and $c_{E_8^{\oplus 2}}$ (cf. Thm. \ref{propbb}). On the other hand, the GIT boundary consists of two points corresponding to the orbits of the curves $C_{D}$ and $C_{2A_5}$ and a $1$-dimensional boundary component corresponding to the orbits of curves $C_{A,B}$ (cf. Thm. \ref{thmcubic}). A standard computation with limit Hodge structures (see \cite{act}, \cite[Rem. 5.6]{friedman} for related computations), based on the fact that $D_4$ and $A_5$ give, via the triple
 cover, the boundary singularities $\widetilde{E}_6$ and $\widetilde{E}_8$ respectively,  allows us to match the boundary components as follows:
\begin{eqnarray*}
\overline{M}_4^{GIT}&\dashrightarrow& (\calB_9/\Gamma)^*\\
C_{D}&\to&c_{E_6^{\oplus 2}\oplus A_2^{\oplus 2}}\\
C_{2A_5}&\to&c_{E_8^{\oplus 2}}\\
C_{A,B}&\to&c_{E_6^{\oplus 2}\oplus A_2^{\oplus 2}}.
\end{eqnarray*}
The full strength of Theorem \ref{mainthm} says that in fact the period map extends to an isomorphism at the boundary points corresponding to $C_{D}$ and $C_{2A_5}$. Furthermore, the period map extends along the boundary curve $Z\subset \overline{M}_4^{GIT}$ parameterizing the curves $C_{A,B}$, except at the point $o\in Z$ corresponding to the orbit of the ribbon curve. The point $o$ is blown-up to introduce the hyperelliptic divisor $H_h=\calH_h/\Gamma$. Finally, the strict transform of the curve $Z$ is contracted to the cusp $c_{E_6^{\oplus 2}\oplus A_2^{\oplus 2}}$.

\bibliography{g4bib}
\end{document}